%% file: write-manuscript.tex
\documentclass[preprint]{imsart}

\RequirePackage{amsthm,amsmath,amsfonts,amssymb}
\RequirePackage[numbers,sort&compress]{natbib}
\RequirePackage[colorlinks,citecolor=blue,urlcolor=blue]{hyperref}
\RequirePackage{graphicx}
\RequirePackage{enumerate}
\RequirePackage{pgf}
\RequirePackage{array,booktabs,tabularx}

\startlocaldefs
\theoremstyle{plain}

\newtheorem{theorem}{Theorem}[section]

\theoremstyle{definition}
\newtheorem{definition}[theorem]{Definition}

\newcommand{\R}{\mathbb{R}}

\newcommand{\E}{\mathbb{E}}

\newcommand{\Var}{\operatorname{Var}}
\newcommand{\Cov}{\operatorname{Cov}}

\newcommand{\dd}{\,\mathrm{d}}
\newcommand{\Normal}{\mathcal{N}}
\newcommand{\Emp}{\widehat P_n}
\newcommand{\fhat}{\widehat f}

\newcommand{\calE}{\mathcal{E}}

\newcommand{\calM}{\mathcal{M}}
\newcommand{\calL}{\mathcal{L}}

\newcommand{\AMISE}{\operatorname{AMISE}}
\newcommand{\tr}{\operatorname{tr}}

\endlocaldefs

\begin{document}

\begin{frontmatter}
\title{Density Evolution: A Multiscale View of Density Estimation}
\runauthor{K. You}
\runtitle{Density Evolution}

\begin{aug}
%%%%%%%%%%%%%%%%%%%%%%%%%%%%%%%%%%%%%%%%%%%%%%%
%% Additional information such as            %%
%% identifying the corresponding author must %%
%% be included in in the Acknowledgments     %%
%% section if necessary.                     %%
\author[A,B,C]{\fnms{Kisung}~\snm{You}}

\address[A]{Department of Mathematics, Baruch College}
\address[B]{Department of Mathematics, The Graduate Center, City University of New York}
\address[C]{Department of Psychiatry, Yale University}

%\address[B]{Name2 Surname2 is Professor, Department,University or Company Name, City, Country\printead[presep={\ }]{e2}.}

%\address[C]{Name3 Surname3 is Distinguished Professor, Department,University or Company Name, City, Country\printead[presep={\ }]{e3,u1}.}

\end{aug}

\begin{abstract}
Density estimation is often presented as a choice among parametric summaries, finite mixtures, and nonparametric smoothers. 
This review argues for a complementary view: a data set can be studied through a path of densities indexed by smoothing scale, diffusion time, model complexity, density level, or noise level. We call this perspective density evolution. Under this lens, Gaussian kernel density estimation is heat flow from the empirical measure; scale-space methods, critical bandwidths, mode trees, and derivative-significance displays describe the evolution of modal and derivative structure; finite mixtures and mixture reduction provide compressed representations of kernel-like estimates; and cluster trees and persistent homology summarize evolving level-set topology. We review these connections and discuss inference for feature lifetimes, high-dimensional complications, and links with score-based generative diffusion. We also include three elementary structural results: nondegenerate modes move along smooth branches, a natural moment-preserving Gaussianization semigroup is forced to be Ornstein--Uhlenbeck, and shared-covariance Gaussian mixtures become log-concave once component means are sufficiently concentrated. Together, these ideas shift attention from choosing one density estimate to studying the multiscale probability landscape.
\end{abstract}

\begin{keyword}
\kwd{density estimation}
\kwd{KDE}
\kwd{scale space}
\kwd{Gaussian mixtures}
\kwd{mode trees}
\kwd{diffusion processes}
\kwd{topological data analysis}
\end{keyword}

\end{frontmatter}

%----------------------------------------------------------------------------
\section{Introduction}

Density estimation offers many summaries of a sample. Common choices include a Gaussian distribution, a finite mixture, a kernel density estimate, a cluster tree, and a topological summary. These objects are often presented as competing answers to one question. A finite data set, however, rarely has one natural resolution. It may appear unimodal at a coarse scale, multimodal at an intermediate scale, and irregular at a fine scale.

We use \emph{density evolution} as an organizing phrase for studying a family of probability landscapes indexed by scale, complexity, level, or noise. The phrase does not define a new field. It connects established work on kernel estimation, Gaussian scale space, mixtures, density clustering, topology, diffusion, and score-based generation. These indices do not define one common mathematical structure. Each path needs its own operator, target, ordering, and inferential interpretation.

Let $X_1,\ldots,X_n$ be independent observations from a probability measure $P$ on $\R^d$, and let $\Emp=n^{-1}\sum_i\delta_{X_i}$. For the Gaussian density $\varphi_{h^2I_d}$ with covariance $h^2I_d$, define
\begin{equation}
  f_h=\varphi_{h^2I_d}*P,
  \qquad
  \fhat_h=\varphi_{h^2I_d}*\Emp.
  \label{eq:gaussian-kde-intro}
\end{equation}
Equation~\eqref{eq:gaussian-kde-intro} gives $\E\{\fhat_h(x)\}=f_h(x)$ for fixed $h>0$ and $x$. An empirical mode or level set therefore estimates a feature of the smoothed population target $f_h$. Interpreting it as a feature of an unsmoothed density requires additional bias, smoothness, and separation conditions.

Applications usually report bandwidth $h$. The additive heat coordinate is $t=h^2$. Writing $u_t=f_{\sqrt t}$ and $\widehat u_t=\fhat_{\sqrt t}$ gives
\begin{equation*}
  \partial_t u_t=\frac12\Delta u_t,
  \qquad
  \partial_t\widehat u_t=\frac12\Delta\widehat u_t,
\end{equation*}
with $\widehat u_t\Rightarrow\Emp$ as $t\downarrow0$. Here $\Delta g=\tr\{\nabla^2g\}$ is the Laplacian. This heat-flow identity is established scale-space theory, not a new result \citep{koenderink_1984_StructureImages,lindeberg_1994_ScalespaceTheoryBasic}.

Other paths answer different questions. A mixture-reduction path changes representational complexity. An Ornstein--Uhlenbeck (OU) path can preserve the first two moments while approaching a fitted Gaussian. A cluster tree varies density level at a fixed density. Persistent homology records topology along a declared filtration. Score-based generative methods learn reverse dynamics along a noise path. Results about the Gaussian semigroup or one-dimensional modal causality do not transfer automatically among these constructions.

The classical roots lie in nonparametric density estimation \citep{rosenblatt_1956_RemarksNonparametricEstimates,parzen_1962_EstimationProbabilityDensity,wand_1994_KernelSmoothing,scott_2015_MultivariateDensityEstimation}. Gaussian scale space, critical bandwidths, mode trees, mode forests, and SiZer developed the multiscale modal view \citep{witkin_1983_ScalespaceFiltering,silverman_1981_UsingKernelDensity,minnotte_1993_ModeTreeTool,minnotte_1998_BumpyRoadModeForest,chaudhuri_1999_SiZerExplorationStructures}. Mixture, cluster-tree, and persistence literatures provide distinct paths and feature summaries \citep{lindsay_1995_MixtureModelsTheory,hartigan_1975_ClusteringAlgorithms,edelsbrunner_2002_TopologicalPersistenceSimplification}. Existing reviews cover modal statistics and related software more directly \citep{chacon_2020_ModalAgeStatistics,ameijeiras-alonso_2021_multimode}.

This article is a theory-focused narrative synthesis, not a systematic review. It isolates the operators, targets, feature maps, and order structures that underlie several multiscale constructions. The literature is selected to establish those connections and their limitations. Broader links to ridges, adaptive diffusion, high-dimensional estimation, and generative diffusion remain part of the conceptual scope.

The main contribution is organizational and mathematical. The framework distinguishes population paths from empirical paths, separates bandwidth from heat time, and records when feature collections do or do not form filtrations. These distinctions prevent structural properties of one evolution from being transferred to another without a defined map or operator.

The article contains three elementary structural results. The first is a local implicit-function description of nondegenerate critical-point motion, with prior one-dimensional treatment by Minnotte \citep{minnotte_1997_NonparametricTestingExistence}. The second characterizes an OU Gaussianization semigroup under explicit full-rank and moment assumptions. The third gives a sufficient log-concavity condition for shared-covariance Gaussian mixtures. These results formalize limited structural statements and do not supply global branch matching or inferential calibration.

Section~\ref{sec:framework} defines paths, targets, and feature summaries. Sections~\ref{sec:kde}--\ref{sec:diffusion} review smoothing, modes, mixtures, and diffusion. Section~\ref{sec:levelsets} corrects the relation between bandwidth and density-level filtrations. Section~\ref{sec:inference} separates structural, descriptive, inferential, and stability claims. Later sections discuss high dimension, generative diffusion, mathematical limitations, and open problems. The unified appendix contains all proofs and a counterexample to cross-bandwidth nesting.

%----------------------------------------------------------------------------
\section{A unifying framework}\label{sec:framework}

A density evolution consists of an initial measure, an indexed operator, and a feature map. The operator may act on the population distribution or on its empirical version. The feature map extracts the object to be interpreted.

\begin{definition}[Density evolution]
Let $\mathcal P$ be a class of probability measures on $\R^d$ and let $T$ be a partially ordered index set. A density evolution is a family of maps
\[
  \calE_t:\mathcal P\to \mathcal G,
  \qquad t\in T,
\]
where $\mathcal G$ is a class of densities, generalized densities, or real-valued functions. The population and empirical evolutions are
\[
  f_t=\calE_t(P),
  \qquad
  \fhat_t = \calE_t(\Emp),
  \qquad t\in T.
\]
A feature evolution is obtained by applying a target-specific feature map $\Phi$:
\[
  \Phi(f_t),
  \qquad
  \Phi(\fhat_t),
  \qquad t\in T.
\]
\end{definition}

The index set may be one-dimensional, such as bandwidth $h$, heat time $t$, or a component count $K$. A collection may also be indexed by bandwidth and density level. Multiple indices alone do not provide the inclusion maps required for a multiparameter persistence module. Common examples are summarized in Table~\ref{tab:evolution-taxonomy}.

Bandwidth carries measurement units. Centering, scaling, or otherwise transforming coordinates changes isotropic smoothing geometry and belongs in the target specification. Heat time $t=h^2$ is useful for semigroup calculations, while applications may remain easier to interpret on the $h$ scale.

\begin{table*}[t]
\centering
\caption{A taxonomy of density evolutions. The same sample can generate several paths, each emphasizing a different meaning of scale or complexity.}
\label{tab:evolution-taxonomy}
%\scriptsize
\begin{tabular}{p{0.17\linewidth}p{0.15\linewidth}p{0.28\linewidth}p{0.26\linewidth}}
\hline
Family & Index & Evolution operator & Typical features \\
\hline
Kernel smoothing
& bandwidth $h$
& convolution $\widehat P_n * K_h$
& modes, derivative signs, bumps \\
Heat flow
& time $t$
& Markov semigroup $P_t^*$
& entropy, Fisher information, modes \\
Adaptive diffusion
& time $t$
& spatially varying diffusion
& boundary-aware and local-scale features \\
Mixture path
& component number $K$ or penalty
& likelihood, merging, pruning
& mixture components, modal structure \\
Gaussianization
& time $t$
& Ornstein--Uhlenbeck-type flow
& path from empirical measure to fitted Gaussian \\
Cluster tree
& level $\lambda$
& superlevel sets $\{x:f(x)\ge\lambda\}$
& high-density clusters, merge events \\
Topological persistence
& filtration
& sublevel or superlevel sets
& components, loops, voids \\
Score diffusion
& noise time $t$
& forward and reverse stochastic differential equations
& score field, generation, denoising \\
\hline
\end{tabular}
\end{table*}

The central statistical shift is from estimating one density to studying target-specific features along a declared path. Let $\calM(h)$ be the set of nondegenerate modes of $\fhat_h$, and
\[
  N(h)=|\calM(h)|
\]
be the corresponding mode count. The empirical path asks where a matched mode is observed across bandwidths. This is distinct from inference on the modal count of an unsmoothed density. Similarly, at a fixed bandwidth the superlevel sets
\[
  \widehat L_{h,\lambda}=\{x:\fhat_h(x)\ge \lambda\}
\]
are nested as $\lambda$ decreases. Their components form an empirical cluster tree at that bandwidth. Varying $h$ produces a collection of such trees, not generally one joint filtration.

\subsection{Feature maps and events}
\label{subsec:feature-maps}

Several feature maps recur throughout the review. For smooth $f$, the displayed mode map records only nondegenerate interior modes.
\begin{align*}
  \Phi_{\mathrm{mode}}(f) &= \{x:\nabla f(x)=0,\ \nabla^2 f(x)\prec 0\},\\
  \Phi_{\mathrm{ridge}}(f) &= \text{density ridges or filamentary sets},\\
  \Phi_{\mathrm{level}}(f,\lambda) &= \pi_0\{x:f(x)\ge\lambda\},\\
  \Phi_{\mathrm{top}}(f,\lambda) &= \{H_k(\{x:f(x)\ge\lambda\}):k\ge0\},\\
  \Phi_{\mathrm{score}}(f) &= \nabla\log f.
\end{align*}
Here $\pi_0$ denotes connected components and $H_k$ denotes homology. The corresponding events include births, deaths, mergers, splits, and changes in qualitative derivative sign.

Feature duration needs a coordinate and a finite observation range. On $\mathcal H=[h_{\min},h_{\max}]$, a matched branch is reported by its observed existence interval $[h_b,h_d]$. An endpoint at $h_{\min}$ or $h_{\max}$ is left- or right-censored. The length $h_d-h_b$ changes after replacing $h$ by $h^2$ or $\log h$, so it is not an invariant measure of importance. On a general partially ordered set, the appropriate summary is the set of indices where the feature occurs. Subtraction of birth and death indices need not be defined.

\subsection{Four principles}
\label{subsec:principles}

The following principles guide the synthesis.

\begin{enumerate}[(1)]
\item \textbf{Scale is information, not nuisance.} Bandwidth selection is important, but a selected bandwidth can discard the information contained in the rest of the path.
\item \textbf{Feature existence is descriptive until calibrated.} A long branch can be informative, but its span alone does not establish significance. A short branch may be noise or a narrow population feature.
\item \textbf{Different summaries see different geometry.} Modes, mixture components, level-set components, ridges, and topological features are related but not interchangeable.
\item \textbf{Sampling uncertainty applies to the movie.} Since $\fhat_t$ is random for every $t$, all events along the path are random objects.
\end{enumerate}

These principles are already present in parts of the literature. SiZer assesses derivative signs across an eligible location-scale grid. Mode trees display observed branch intervals. Cluster trees and persistent homology encode density-level structure at a fixed density. The density-evolution perspective compares these objects without making their claims interchangeable.

%----------------------------------------------------------------------------
\section{Kernel density estimation as evolution}
\label{sec:kde}

Kernel density estimation is the simplest density evolution. Let $K$ be a kernel on $\R^d$ with $\int K(u)\dd u=1$ and define $K_h(x)=h^{-d}K(x/h)$. The population target and its empirical estimator are
\begin{equation*}
  f_h=K_h*P,
  \qquad
  \fhat_h=K_h*\Emp
  =\frac{1}{n}\sum_{i=1}^n K_h(\,\cdot-X_i).
\end{equation*}
At every fixed location, $\E\{\fhat_h(x)\}=f_h(x)$. This identity uses the linear convolution operator and should not be assumed for nonlinear mixture-fitting or clustering procedures. In the usual point-estimation view, $h$ is chosen by cross-validation, plug-in rules, normal-reference rules, or risk approximations. In the evolution view, $h$ indexes a family of increasingly smooth targets and estimates.

\begin{figure*}[t]
  \centering
  \includegraphics[width=\textwidth]{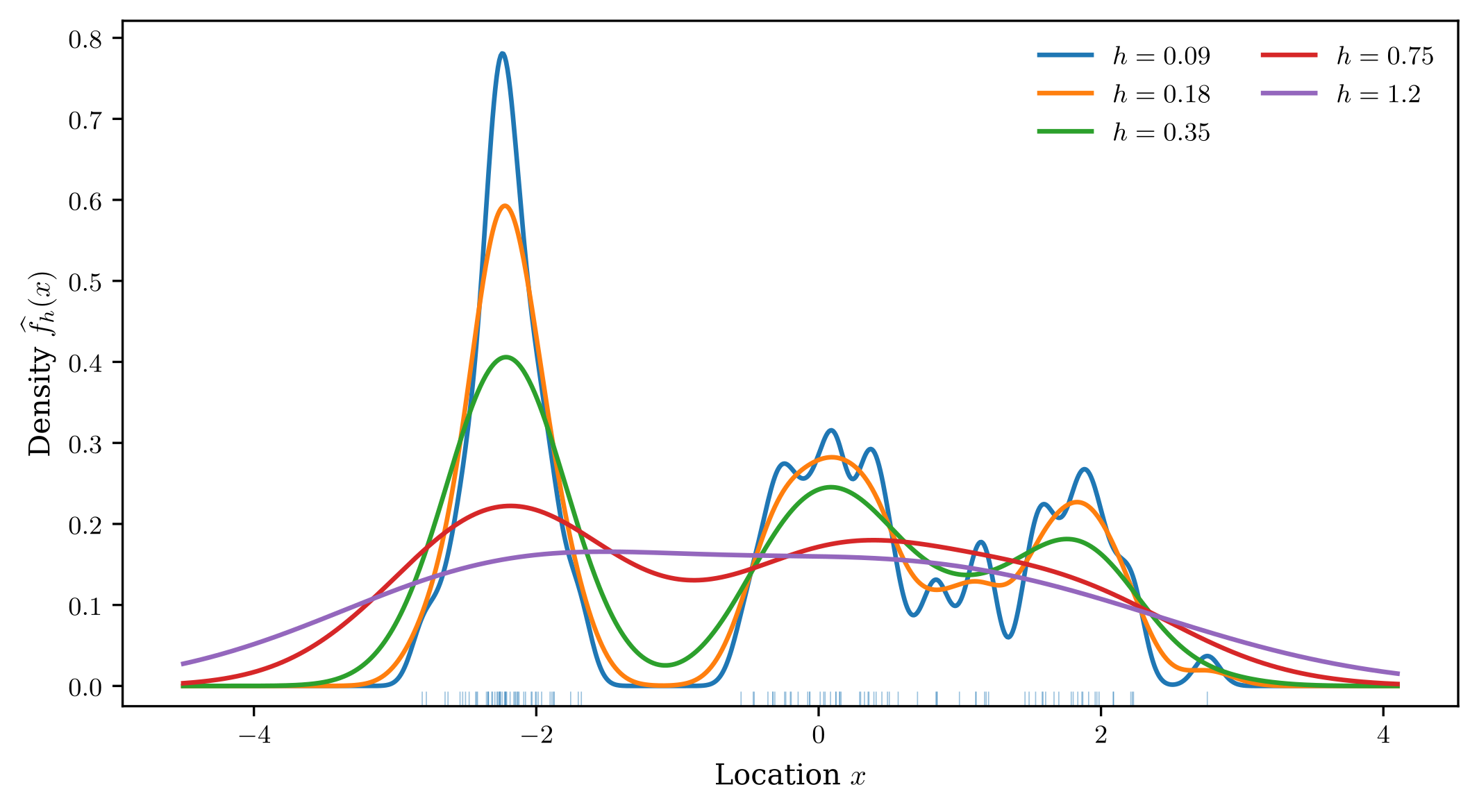}
  %\resizebox{\textwidth}{!}{\input{figures/kde_path.pgf}}
  \caption{An illustrative density path generated by changing the bandwidth of a Gaussian KDE. Small bandwidths reveal sample-level detail, while large bandwidths produce a coarse summary. Intermediate branches are descriptive until their sampling behavior is calibrated.}
  \label{fig:kde-path}
\end{figure*}

\subsection{Bias, variance, and the one-bandwidth tradition}
\label{subsec:bias-variance}

Suppose $P$ has density $f$ and $K$ is symmetric with second moment matrix $\mu_2(K)I_d$. Then
\[
  \E\{\fhat_h(x)\}=(K_h*f)(x).
\]
If $f$ is smooth, Taylor expansion gives
\begin{equation*}
  \E\{\fhat_h(x)\}-f(x)
  = \frac{h^2\mu_2(K)}{2}\Delta f(x)+o(h^2).
\end{equation*}
Under standard regularity conditions, the pointwise variance satisfies
\begin{equation*}
  \Var\{\fhat_h(x)\}
  = \frac{1}{nh^d}R(K)f(x)+o\{(nh^d)^{-1}\},
\end{equation*}
for $R(K)=\int K^2(u)\dd u$.
Integrated over $x$, this leads to the usual asymptotic mean integrated squared error (AMISE) approximation
\begin{equation}
  \AMISE(h)
  = \frac{h^4\mu_2(K)^2}{4}R(\Delta f)+\frac{R(K)}{nh^d},
  \label{eq:amise}
\end{equation}
where $R(g)=\int g^2$. The minimizer obeys
\begin{equation}
  h_{\AMISE}\asymp n^{-1/(d+4)}.
  \label{eq:amise-rate}
\end{equation}
This classical calculation is indispensable, explaining the bias-variance tradeoff and the curse of dimensionality. But it also encourages a one-bandwidth mindset: choose $h$ and then inspect $\fhat_h$. The density-evolution perspective keeps the calculation while changing the interpretation. Equation \eqref{eq:amise} is a risk approximation for one frame of a movie. It does not by itself answer which features persist across frames.

\subsection{Gaussian KDE as heat flow}
\label{subsec:kde-heat}

For the Gaussian kernel, KDE is exactly a diffusion after changing coordinates from bandwidth to heat time. Let
\begin{equation*}
  \varphi_t(x)=(2\pi t)^{-d/2}\exp\left(-\frac{\|x\|^2}{2t}\right),
  \qquad t>0.
\end{equation*}
Define the population and empirical heat paths
\begin{equation*}
  u_t=\varphi_t*P,
  \qquad
  \widehat u_t=\varphi_t*\Emp,
\end{equation*}
where $t=h^2$. Since
\[
  \partial_t\varphi_t=\frac12\Delta\varphi_t,
\]
linearity gives
\begin{equation*}
  \partial_tu_t=\frac12\Delta u_t,
  \qquad
  \partial_t\widehat u_t=\frac12\Delta\widehat u_t,
  \qquad
  \widehat u_t\Rightarrow\Emp\quad\text{as }t\downarrow0.
\end{equation*}
For a smooth function $g$, $\nabla g$ is its gradient, $\nabla^2g$ is its Hessian, and $\Delta g=\tr\{\nabla^2g\}=\sum_j\partial_{jj}g$. Thus a Gaussian KDE is the heat evolution of empirical point masses. This established identity gives a Markov semigroup interpretation of Gaussian smoothing.

The same identity also explains why the Gaussian kernel has a special role in scale-space theory. In image analysis, the Gaussian scale-space representation arises from axioms such as linearity, translation invariance, semigroup structure, and non-enhancement of local extrema in one dimension \citep{witkin_1983_ScalespaceFiltering, koenderink_1984_StructureImages, babaud_1986_UniquenessGaussianKernel, lindeberg_1994_ScalespaceTheoryBasic}. In statistics, the same semigroup gives a family of density estimates indexed by resolution.

\subsection{From density values to density features}
\label{subsec:kde-features}

Classical risk measures such as \eqref{eq:amise} quantify error in density values. Many uses of density estimation, however, are feature-driven: how many modes are present, where are the high-density regions, is there a ridge or filament, and which observations lie in low-density regions? A density path makes these questions scale-dependent.

Denote by $N(h)$ the number of nondegenerate modes of $\fhat_h$. For small $h$, $N(h)$ may be close to $n$ or may reflect sample-level irregularity. For large $h$, $N(h)$ may be one. A bandwidth chosen to minimize integrated squared error need not best reveal modal structure. Multimodality testing, bump hunting, and scale-space methods therefore require targets and calibrations beyond classical bandwidth selection.

%----------------------------------------------------------------------------
\section{Scale space, critical bandwidths, mode trees, and SiZer}
\label{sec:scale-space}

Scale space treats smoothing scale as a coordinate. Instead of plotting one curve $x\mapsto\fhat_h(x)$, one studies the surface
\[
  (x,h)\mapsto\fhat_h(x)
\]
or its derivatives. This idea is old in computer vision and signal processing \citep{witkin_1983_ScalespaceFiltering, koenderink_1984_StructureImages, lindeberg_1994_ScalespaceTheoryBasic}, but it has a distinctly statistical form in critical-bandwidth testing, mode trees, and SiZer.

\subsection{Critical bandwidths and multimodality}
\label{subsec:critical-bandwidth}

For a univariate Gaussian KDE, let $N(h)$ count nondegenerate modes and define the critical bandwidth
\begin{equation*}
  \widehat h_k = \inf\{h>0:N(h)\leq k\}.
\end{equation*}
Under the standard one-dimensional variation-diminishing conditions, $N(h)$ is nonincreasing in $h$ \citep{schoenberg_1951_PolyaFrequencyFunctions,silverman_1981_UsingKernelDensity}. Nongeneric shoulders can mark transitions. Multivariate Gaussian smoothing can create modes, so the merger-only interpretation does not extend to higher dimensions \citep{carreira-perpinan_2003_NumberModesGaussian}.

A critical bandwidth records the empirical smoothing needed to satisfy a modal restriction. It becomes a test statistic only after calibration. Composite modal nulls and shoulder densities make calibration delicate. Corrected critical-bandwidth and excess-mass procedures address limitations of the original bootstrap \citep{hall_2001_CalibrationSilvermansTest,fisher_2001_ModeTestingExcess,ameijeiras-alonso_2019_ModeTestingCritical}.

The ACR procedure tests exactly $k$ population modes against more than $k$. Nonrejection establishes neither an exact count nor an upper bound. Its implementations for $k=1$ and $k>1$ use different statistics, which should not be compared directly. Prespecified multiple tests also require multiplicity adjustment.

\subsection{Mode trees}
\label{subsec:mode-trees}

The mode tree plots estimated mode locations as a function of bandwidth \citep{minnotte_1993_ModeTreeTool}. In one dimension, define
\[
  \calM(h)=\{x:\fhat_h'(x)=0,\ \fhat_h''(x)<0\}
\]
in one dimension. The mode tree displays the set
\begin{equation*}
  \mathcal T_{\mathrm{mode}}=\{(x,h):x\in\calM(h)\}.
\end{equation*}
A branch's observed span is descriptive. Long existence does not establish significance. Short existence can reflect sampling noise or a narrow population feature. Figure~\ref{fig:mode-tree} illustrates this display. Its features are critical points of a smoothed estimate, not homology classes of a density-level filtration.

\begin{figure}[t]
  \centering
  \includegraphics[width=.55\textwidth]{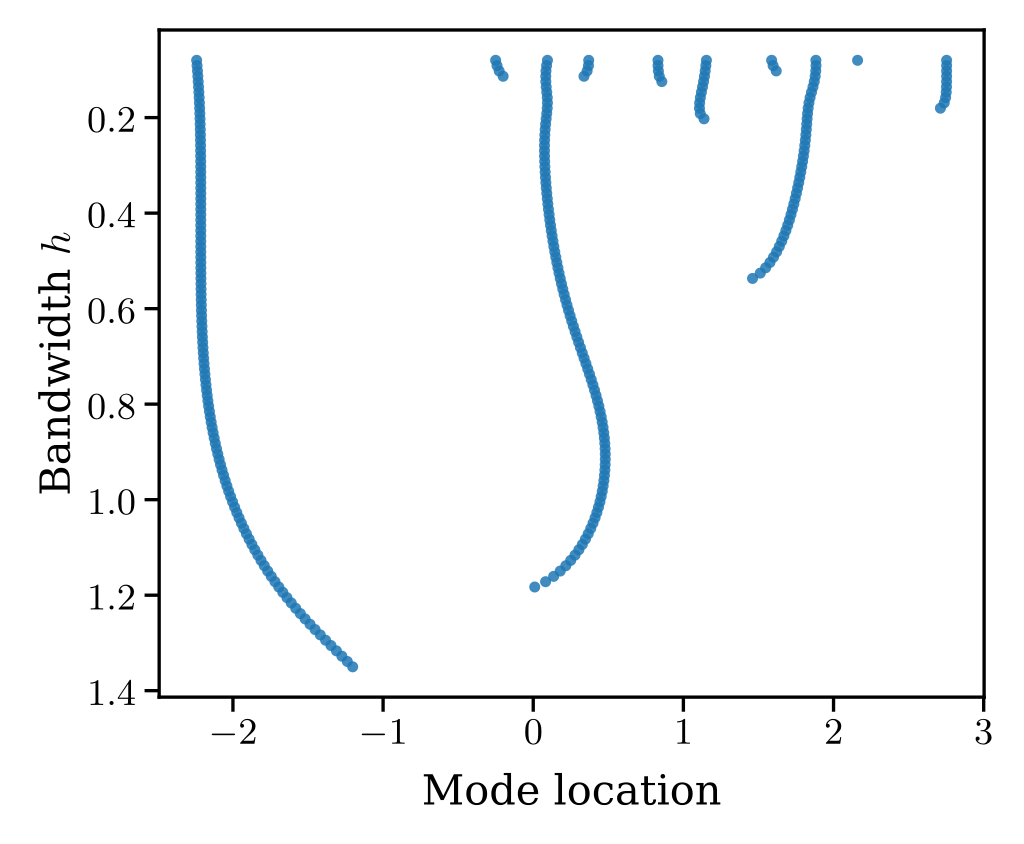}
  %\resizebox{.55\columnwidth}{!}{\input{figures/mode_tree.pgf}}
  \caption{An illustrative mode tree for the sample in Figure~\ref{fig:kde-path}. Each point marks an estimated mode at one bandwidth. The vertical axis is inverted, so fine scales appear near the top. Branch span alone is not an inferential significance measure.}
  \label{fig:mode-tree}
\end{figure}

Mode trees encourage a different interpretation of multimodality. One asks where empirical branches exist, whether their endpoints are censored by the displayed range, and which separate procedure can calibrate a population claim. In multivariate settings, modes can move through space, saddles and ridges matter, and modal counts need not change monotonically under smoothing \citep{carreira-perpinan_2003_NumberModesGaussian,ray_2005_TopographyMultivariateNormal}.

A mode forest shades location-bandwidth pixels by the fraction of bootstrap samples containing a mode there \citep{minnotte_1998_BumpyRoadModeForest}. This quantity is pixelwise modal occupancy. It is not the probability of a matched complete branch and is not a coverage probability. Branch-specific summaries require an explicit matching rule and treatment of unmatched or censored events.

The following theorem records the local implicit-function structure behind a mode tree. A one-dimensional treatment appears in Minnotte \citep{minnotte_1997_NonparametricTestingExistence}. The result does not give global continuation, statistical calibration, or a matching rule for modes recomputed on a discrete bandwidth grid.

\begin{theorem}[Local motion of modes]
\label{thm:mode-motion}
Let $U\subset\R^d$ be open and let $I\subset\R$ be an interval.
Suppose $f:U\times I\to\R$ is twice continuously differentiable in $x$,
differentiable in $t$, and that
\[
  F(x,t)=\nabla_x f(x,t)
\]
is continuously differentiable in $(x,t)$. Assume also that mixed
derivatives commute, so that
\[
  \partial_t\nabla_x f=\nabla_x\partial_t f.
\]
Let $(x_0,t_0)\in U\times I$ satisfy
\[
  \nabla_x f(x_0,t_0)=0,
  \qquad
  H_0:=\nabla_x^2 f(x_0,t_0) \text{ is nonsingular}.
\]
Then there are neighborhoods $V$ of $x_0$ and $J$ of $t_0$, and a unique
continuously differentiable curve $m:J\to V$, such that
\[
  m(t_0)=x_0,
  \qquad
  \nabla_x f(m(t),t)=0
  \quad(t\in J).
\]
Moreover,
\[
  \dot m(t)
  =-\{\nabla_x^2 f(m(t),t)\}^{-1}
     \nabla_x\{\partial_t f(m(t),t)\}.
\]
If $H_0$ is negative definite, then after possibly shrinking $J$, the
curve consists of local modes. For heat flow,
\(\partial_t f_t=(1/2)\Delta f_t\), this becomes
\[
  \dot m(t)
  =-\frac12\{\nabla_x^2 f_t(m(t))\}^{-1}
    \nabla_x\Delta f_t(m(t)).
\]
Consequently, fix \(t_0\in I\) and a compact set \(K\subset U\). If no
critical point of \(f(\cdot,t_0)\) lies on \(\partial K\), and all
critical points of \(f(\cdot,t_0)\) in \(K\) are nondegenerate, then for
all \(t\) sufficiently close to \(t_0\), the number and Morse index of
critical points in \(K\) are unchanged.
\end{theorem}

The Morse index is the number of negative Hessian eigenvalues. The theorem is local. A branch can leave the declared domain or plotting range, and empirical branch matching can still be ambiguous near a degeneracy.

\subsection{SiZer and derivative signs across scale}
\label{subsec:sizer}

SiZer assesses derivative signs of a smoothed target across location and scale \citep{chaudhuri_1999_SiZerExplorationStructures,chaudhuri_2000_ScaleSpaceView}. A supported positive region followed by a negative region indicates an increasing-then-decreasing pattern. It need not locate a zero crossing precisely or identify the modal count of an unsmoothed density.

A simplified version of the inferential calculation is as follows. Let
\[
  \fhat_h'(x)=\frac{1}{nh^2}\sum_{i=1}^n K'\left(\frac{x-X_i}{h}\right)
\]
in one dimension. Under regularity conditions, a standardized statistic
\[
  Z_h(x)=\frac{\fhat_h'(x)-\E\{\fhat_h'(x)\}}{\widehat{\operatorname{se}}\{\fhat_h'(x)\}}
\]
is approximately Gaussian. SiZer uses simultaneous calibration to color eligible pixels where the derivative is positive, negative, or unresolved. The exact guarantee depends on the method, effective-sample-size rule, finite grid, and declared range.

Finite-grid calibration does not automatically cover a continuum or account for a bandwidth range selected from the same data. An aligned mode tree and SiZer map therefore provide complementary descriptive and derivative evidence.

\subsection{Causality and the multivariate caveat}
\label{subsec:causality}

A guiding principle in scale-space theory is causality: moving to coarser
scales should not create new meaningful structure. In one dimension, Gaussian smoothing has a variation-diminishing property related to total positivity, and the number of modes cannot increase under Gaussian convolution under suitable conditions \citep{schoenberg_1951_PolyaFrequencyFunctions, silverman_1981_UsingKernelDensity}. This justifies the interpretation that increasing bandwidth merges or removes modes.

In dimensions $d\ge2$, the picture is more delicate. Gaussian blurring can create modes for general functions, and the number of modes of a Gaussian mixture can behave in surprising ways \citep{carreira-perpinan_2003_NumberModesGaussian}. This does not invalidate multiscale density analysis, but it weakens the simple one-dimensional story. A major open question is whether useful classes of multivariate density evolutions can be constructed with a strong causality property for modes, ridges, or topological summaries.

%----------------------------------------------------------------------------
\section{Mean shift, modes, and gradient flows of the KDE}
\label{sec:mean-shift}

The modal structure of a KDE is linked to gradient ascent. This connection is important both computationally and conceptually, because it turns modes into attractors of a vector field.

For the isotropic Gaussian KDE, define
\[
  \mu_h(x)=
  \frac{\sum_{i=1}^nX_i\exp\{-\|x-X_i\|^2/(2h^2)\}}
       {\sum_{i=1}^n\exp\{-\|x-X_i\|^2/(2h^2)\}}.
\]
Direct differentiation gives the exact identity
\begin{equation}
  \mu_h(x)-x = h^2\nabla\log\fhat_h(x).
  \label{eq:gaussian-mean-shift-score}
\end{equation}
This identity links KDE stationary points, score fields, and iterative clustering. Fixed points are stationary points, but a Hessian check is required to identify modes. Mean-shift clustering assigns points to modal basins under suitable conditions \citep{fukunaga_1975_EstimationGradientDensity,cheng_1995_MeanShiftMode,comaniciu_2002_MeanShiftRobust,azzalini_2007_ClusteringNonparametricDensity}.

Each bandwidth has its own gradient field and attraction basins. Mean-shift iteration operates at one fixed bandwidth. It does not traverse smoothing scales or establish a correspondence between stationary points at different bandwidths.

%----------------------------------------------------------------------------
\section{From kernels to mixtures}
\label{sec:mixtures}

Kernel estimates and Gaussian mixtures are often taught separately: the former as nonparametric smoothing, the latter as parametric or semiparametric modeling. Yet the Gaussian KDE is a special Gaussian mixture,
\begin{equation*}
  \fhat_h(x)=\sum_{i=1}^n \frac1n\,\Normal(x\mid X_i,h^2I_d).
\end{equation*}
A general Gaussian mixture is
\begin{equation*}
  f_K(x)=\sum_{k=1}^K \pi_k\,\Normal(x\mid \mu_k,\Sigma_k),
  \quad
  \pi_k\ge0,~ \sum_{k=1}^K\pi_k=1.
\end{equation*}
The KDE corresponds to $K=n$, equal weights, fixed means, and common covariance. A single Gaussian corresponds to $K=1$. This comparison motivates a distinct discrete complexity path.

\subsection{Discrete complexity paths}
\label{subsec:mixture-paths}

Finite mixtures give a discrete path
\[
  K=1,2,3,\ldots,
\]
where $K$ is the number of components. Model selection criteria, penalized likelihood, minimum message length, reversible-jump Markov chain Monte Carlo, and Bayesian nonparametric mixtures all provide ways to move along or average over this path \citep{lindsay_1995_MixtureModelsTheory,richardson_1997_BayesianAnalysisMixtures,figueiredo_2002_UnsupervisedLearningFinite,fraley_2002_ModelBasedClusteringDiscriminant,mclachlan_2019_FiniteMixtureModels}. Bandwidth controls smoothing continuously, while $K$ controls representational complexity discretely. Their numerical coordinates and feature durations are not commensurate.

The phrase ``component'' must be used carefully. A mixture component is a representational object, while a mode is a feature of the resulting density. A mixture with $K$ Gaussian components can have more or fewer than $K$ modes, depending on dimension and covariance structure \citep{carreira-perpinan_2003_NumberModesGaussian,ray_2005_TopographyMultivariateNormal}.

\subsection{Kernel-to-mixture reduction}\label{subsec:kernel-to-mixture}

A particularly close precursor is the kernel-to-mixture reduction work of Scott and Szewczyk \citep{scott_2001_KernelsMixtures}. It treats kernel estimators as large mixtures and studies how to collapse components recursively to obtain a more parsimonious mixture representation. At a high level, one starts with a KDE-like mixture and repeatedly merges components that are close according to a density-similarity measure. If two Gaussian components with weights $\pi_a,\pi_b$, means $\mu_a,\mu_b$, and covariances $\Sigma_a,\Sigma_b$ are merged, a moment-matched replacement has weight
\[
  \pi=\pi_a+\pi_b,
\]
mean
\[
  \mu=\frac{\pi_a\mu_a+\pi_b\mu_b}{\pi},
\]
and covariance
\begin{align*}
  \Sigma
  &=\frac{\pi_a}{\pi}\left\{\Sigma_a+(\mu_a-\mu)(\mu_a-\mu)^T\right\}\\
  &+\frac{\pi_b}{\pi}\left\{\Sigma_b+(\mu_b-\mu)(\mu_b-\mu)^T\right\}.  
\end{align*}
This operation preserves the first two moments of the pair and creates a path from many components to fewer components. Related ideas appear in alternating kernel and mixture estimates \citep{priebe_2000_AlternatingKernelMixture} and in broader mixture-reduction algorithms.

It is worth emphasizing that this path differs from heat flow. Heat flow increases component variance while keeping component centers fixed. Mixture reduction changes the representation by merging centers and covariances. Both can be viewed as density evolutions, but they encode different notions of simplification.

\subsection{Mixing distributions}
\label{subsec:mixing-distributions}

A general mixture can be written as
\begin{equation*}
  f_G(x)=\int \Normal(x\mid \theta)\,\dd G(\theta),
\end{equation*}
where $G$ is a mixing distribution over parameters $\theta=(\mu,\Sigma)$ or a subset thereof. Then a single Gaussian is obtained when $G$ is a point mass, a finite mixture when $G$ is discrete with $K$ atoms, and a KDE when $G$ is the empirical distribution over locations with fixed covariance. This representation connects KDEs, finite mixtures, nonparametric maximum likelihood for mixtures, and Bayesian nonparametric mixtures  \citep{kiefer_1956_ConsistencyMaximumLikelihood, ferguson_1973_BayesianAnalysisNonparametric, lo_1984_ClassBayesianNonparametric, escobar_1995_BayesianDensityEstimation, lindsay_1995_MixtureModelsTheory}.

The density-evolution viewpoint treats $G$ itself as an object that may evolve. Atoms may split or merge, weights may shrink, covariance may inflate, and the mixing distribution may move from empirical discreteness toward a low-complexity summary.

%----------------------------------------------------------------------------
\section{Diffusion, semigroups, and Gaussianization}
\label{sec:diffusion}

The heat interpretation of Gaussian KDE suggests a broader class of density evolutions based on Markov processes. Let $(X_t)_{t\ge0}$ be a diffusion with generator
\begin{equation*}
\calL\phi(x)=b(x)\cdot\nabla\phi(x)+\frac12\tr\{a(x)\nabla^2\phi(x)\},
\end{equation*}
where $a(x)=\sigma(x)\sigma(x)^T$. If $f_t$ is the density of $X_t$, it satisfies the Fokker--Planck equation
\begin{equation}
  \partial_t f_t
  = -\nabla\cdot\{b f_t\}+\frac12\sum_{i,j}\partial_{ij}\{a_{ij}f_t\}
  = \calL^* f_t.
  \label{eq:fokker-planck}
\end{equation}
Thus a density evolution may be generated by a semigroup $P_t^*$:
\[
  f_t=P_t^* f_0.
\]
For empirical $f_0=\Emp$, this is understood weakly; for $t>0$ the distribution may have a smooth density.

\subsection{Adaptive diffusion}
\label{subsec:adaptive-diffusion}

Classical KDE uses the same smoothing everywhere. Adaptive density estimation modifies the amount of smoothing according to local data density \citep{abramson_1982_BandwidthVariationKernel, terrell_1992_VariableKernelDensity}. Diffusion-based KDE provides a principled partial differential equation (PDE) route. One chooses a diffusion process whose local behavior adapts to a pilot estimate, support constraints, or boundary behavior. Diffusion formulations show how adaptive KDE can be constructed from linear diffusion processes, yielding bandwidth selectors and boundary-aware estimators \citep{botev_2010_KernelDensityEstimation}.

Adaptive diffusion raises a basic question: should scale be global or local? A single bandwidth $h$ imposes one resolution on all regions. A spatially varying diffusion tensor can smooth aggressively in low-information regions and gently near sharp features. The resulting operator is more model-dependent and less transparent.

\subsection{A path from empirical measure to fitted Gaussian}
\label{subsec:ou-path}

Heat flow has an inconvenient endpoint for the original motivation of moving from many local components to one Gaussian: as $t\to\infty$, heat flow spreads mass over all of $\R^d$ rather than converging to the sample Gaussian. An Ornstein--Uhlenbeck-type construction gives a more direct path.

Let
\[
  \bar X=\frac1n\sum_{i=1}^n X_i,
  \qquad
  S=\frac1n\sum_{i=1}^n (X_i-\bar X)(X_i-\bar X)^T.
\]
Assume first that $S$ is positive definite. Let $a_t=e^{-t}$ and define
\begin{equation}
  g_t(x)=\frac1n\sum_{i=1}^n
  \Normal\left(x\mid \bar X+a_t(X_i-\bar X),\ (1-a_t^2)S\right).
  \label{eq:ou-gaussianization}
\end{equation}
At $t=0$, the covariance term vanishes and $g_t$ converges weakly to $\Emp$. For $t>0$, $g_t$ is an $n$-component Gaussian mixture. As $t\to\infty$, all component means collapse to $\bar X$ and the covariance becomes $S$, so
\begin{equation*}
  g_t \to \Normal(\bar X,S).
\end{equation*}
Moreover, the mean and covariance are preserved for all $t$:
\[
  \E_{g_t}(X)=\bar X,
  \qquad
  \Cov_{g_t}(X)=a_t^2S+(1-a_t^2)S=S.
\]
This path can be generated by an OU process after whitening by $S^{-1/2}$. In standardized coordinates, $Y_t=e^{-t}Y_0+\sqrt{1-e^{-2t}}Z$ with $Z\sim\Normal(0,I_d)$. Transforming back gives \eqref{eq:ou-gaussianization}.

The next result characterizes this formulation within a stated class of linear Gaussian evolutions. The OU form is forced by moment preservation and the semigroup property, up to a constant rescaling of time. We state the clean full-rank version. For singular covariance, one may instead work on the empirical affine span and use the induced degenerate Gaussian measure.

The finite-third-moment assumption is used only to identify the coefficient
of the non-Gaussian part through third central moments.
\begin{theorem}[Characterization of the moment-preserving Gaussianization semigroup]
\label{thm:ou-characterization}
Let $\mathcal P_3^+$ be the class of probability measures on $\R^d$ with finite third moments and positive definite covariance. For $P\in\mathcal P_3^+$, write $m_P$ and $S_P$ for its mean and covariance. Suppose a family of maps $(T_t)_{t\ge0}$ has the form
\begin{equation}
  T_tP
  =\mathcal{L}\{m_P+a(t)(X-m_P)+b(t)Z_P\},
  \label{eq:linear-gaussian-class-main}
\end{equation}
where $X\sim P$, $Z_P\sim\Normal(0,S_P)$ is independent of $X$, $a,b$ are continuous functions, $a(0)=1$, $b(0)=0$, $b(t)\ge0$, and $0\le a(t)\le1$. Assume that:
\begin{enumerate}
\item $T_t$ preserves mean and covariance for every $P\in\mathcal P_3^+$;
\item $(T_t)_{t\ge0}$ is a semigroup on $\mathcal P_3^+$, that is, $T_{s+t}P=T_s(T_tP)$ for all $s,t\ge0$ and all $P\in\mathcal P_3^+$;
\item the family is not the identity, meaning $a(t)<1$ for at least one $t>0$.
\end{enumerate}
Then there is a constant $c>0$ such that
\begin{equation*}
  a(t)=e^{-ct},
  \qquad
  b(t)=\sqrt{1-e^{-2ct}}.
\end{equation*}
Thus, after the time change $t\mapsto ct$, \eqref{eq:linear-gaussian-class-main} is exactly the OU Gaussianization path
\[
  T_tP
  =\mathcal{L}\{m_P+e^{-t}(X-m_P)+\sqrt{1-e^{-2t}}Z_P\}.
\]
\end{theorem}

A second structural fact is that shared-covariance Gaussian mixtures have a simple log-concavity certificate. This gives explicit terminal times after which heat-flow and OU density evolutions can no longer have multiple modes or nonconvex superlevel sets. The theorem is deliberately stated for general unequal weights because the proof uses only the common-covariance mixture geometry.

\begin{theorem}[Terminal log-concavity for shared-covariance mixtures]
\label{thm:terminal-logconcavity}
Let
\begin{equation*}
  f(x)=\sum_{i=1}^n \pi_i\,\varphi_\Sigma(x-\mu_i),
  \quad
  \pi_i>0, ~\sum_{i=1}^n\pi_i=1,~
  \Sigma\succ0.
\end{equation*}
Define posterior weights
\[
  w_i(x)=\frac{\pi_i\varphi_\Sigma(x-\mu_i)}{f(x)},
  \qquad
  \bar\mu_w(x)=\sum_i w_i(x)\mu_i,
\]
and the weighted covariance matrix of component means
\[
  C_w(x)=\sum_i w_i(x)\{\mu_i-\bar\mu_w(x)\}\{\mu_i-\bar\mu_w(x)\}^T.
\]
Then
\begin{equation*}
  \nabla^2\log f(x)
  =-\Sigma^{-1}+\Sigma^{-1}C_w(x)\Sigma^{-1}.
\end{equation*}
Consequently, $f$ is log-concave if $C_w(x)\preceq\Sigma$ for all $x$, and strictly log-concave if $C_w(x)\prec\Sigma$ for all $x$. In particular, if there is a point $c\in\R^d$ such that
\begin{equation*}
  \max_i \|\Sigma^{-1/2}(\mu_i-c)\|^2<1,
\end{equation*}
then $f$ is strictly log-concave.

The following consequences hold.
\begin{enumerate}
\item For the heat-flow KDE
\[
  \widehat u_t(x)=\frac1n\sum_{i=1}^n\varphi_{tI_d}(x-X_i),
\]
if $\max_i\|X_i-c\|^2\le R^2$, then $\widehat u_t$ is strictly log-concave for every $t>R^2$.
\item For the OU path \eqref{eq:ou-gaussianization}, assume $S\succ0$ and define
\[
  R_S^2=\max_i (X_i-\bar X)^T S^{-1}(X_i-\bar X).
\]
Then $g_t$ is strictly log-concave whenever
\begin{equation*}
  t>\frac12\log(1+R_S^2).
\end{equation*}
\end{enumerate}
At those times the density has a unique mode, and all positive-level
superlevel sets are convex.
\end{theorem}

\begin{figure*}[t]
  \centering
  \includegraphics[width=\textwidth]{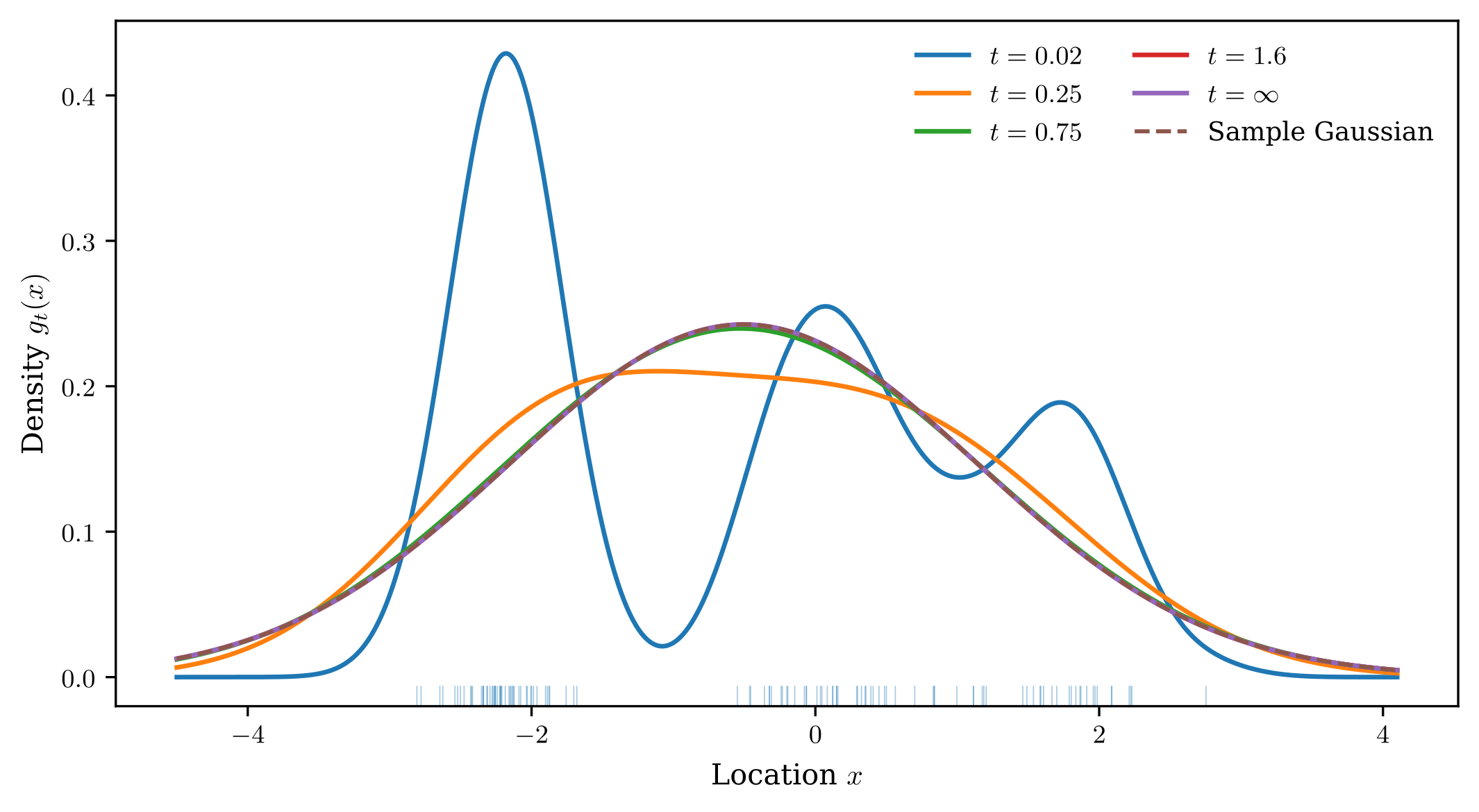}
  %\resizebox{\textwidth}{!}{\input{figures/ou_path.pgf}}
  \caption{An Ornstein--Uhlenbeck-type Gaussianization path. Unlike ordinary heat flow, this path preserves the sample mean and covariance and converges to the fitted Gaussian. It provides a mathematically explicit continuum from empirical spikes to an $n$-component mixture to a single Gaussian summary.}
  \label{fig:ou-path}
\end{figure*}

Figure~\ref{fig:ou-path} illustrates this conceptual construction. It is not an inferential procedure. Under the full-rank assumptions, the path preserves the first two moments and is equivariant under invertible affine maps. Working on the empirical affine span retains those properties when $S$ is singular. Regularizing to $S_\epsilon=S+\epsilon I_d$ defines a different full-dimensional path because
\[
  \Cov(g_t^\epsilon)=S+(1-e^{-2t})\epsilon I_d.
\]
It therefore does not preserve the sample covariance. Whether either construction supports useful model selection or calibrated feature analysis remains open.

\subsection{Entropy and information dissipation}
\label{subsec:entropy}

Diffusion paths often have monotone information quantities. For heat flow, let $f_t=f_0*\varphi_t$ and define differential entropy
\[
  H(f_t)=-\int f_t\log f_t\dd x
\]
and Fisher information
\[
  I(f_t)=\int \|\nabla\log f_t\|^2 f_t\dd x.
\]
A formal integration-by-parts calculation gives de Bruijn's identity
\begin{equation*}
  \frac{\dd}{\dd t}H(f_t)=\frac12 I(f_t).
\end{equation*}
Indeed,
\begin{align*}
  \frac{\dd}{\dd t}H(f_t)
  &= -\int (\partial_t f_t)\log f_t\dd x \\
  &= -\frac12\int (\Delta f_t)\log f_t\dd x \\
  &= \frac12\int \frac{\|\nabla f_t\|^2}{f_t}\dd x
   = \frac12 I(f_t),
\end{align*}
assuming boundary terms vanish. Thus heat flow increases entropy and smooths information. The OU semigroup has an analogous relative-entropy dissipation identity with respect to its Gaussian invariant distribution  \citep{bakry_2014_AnalysisGeometryMarkov}. The Wasserstein gradient-flow formulation of Fokker--Planck equations provides another geometric interpretation of such evolutions \citep{jordan_1998_VariationalFormulationFokkerPlanck, villani_2009_OptimalTransportOld}.

For statistics, these identities suggest continuous measures of distributional complexity. Instead of counting modes or components, one can summarize the path through integrated Fisher information, entropy production, or changes in relative entropy. These summaries may be less visually immediate than mode trees, but they can be stable in high dimension and connect density estimation to information theory \citep{stam_1959_InequalitiesSatisfiedQuantities, cover_2005_ElementsInformationTheory}.

%----------------------------------------------------------------------------
\section{Level sets, cluster trees, and topology}
\label{sec:levelsets}

Modes summarize local maxima. Many clustering and topological questions are better expressed through level sets. For a density $f$ and a level $\lambda$, define the upper level set
\begin{equation*}
  L_\lambda(f)=\{x\in\R^d:f(x)\ge\lambda\}.
\end{equation*}
As $\lambda$ decreases, components of $L_\lambda(f)$ appear and merge. The collection of these components forms the cluster tree.

\subsection{Cluster trees}
\label{subsec:cluster-trees}

Hartigan's classical density-clustering view defines clusters as high-density connected components \citep{hartigan_1975_ClusteringAlgorithms}. More formally, the cluster tree is
\begin{equation*}
  \mathcal C_f = \{\text{connected components of }L_\lambda(f):\lambda\ge0\},
\end{equation*}
with parent-child relationships induced by set inclusion as $\lambda$ changes. This object records a hierarchy without choosing a fixed number of clusters.

The cluster-tree literature has studied estimation, pruning, and consistency 
\citep{stuetzle_2003_EstimatingClusterTree, stuetzle_2010_GeneralizedSingleLinkage, chaudhuri_2010_RatesConvergenceCluster, chaudhuri_2014_ConsistentProceduresCluster, eldridge_2015_HartiganConsistencyMerge}. Level-set estimation and excess-mass methods provide complementary inferential foundations 
\citep{polonik_1995_MeasuringMassConcentrations, tsybakov_1997_NonparametricEstimationDensity, rigollet_2009_OptimalRatesPlugin, samworth_2010_AsymptoticsOptimalBandwidth, rinaldo_2010_GeneralizedDensityClustering, steinwart_2015_FullyAdaptiveDensitybased}.  Density-based algorithms such as HDBSCAN, a hierarchical density-based clustering method, can also be interpreted as practical cluster-tree approximations  \citep{campello_2015_HierarchicalDensityEstimates}.

At bandwidth $h$, distinguish the population and empirical superlevel sets
\begin{equation}
  L_{h,\lambda}=\{x:f_h(x)\ge\lambda\},
  \qquad
  \widehat L_{h,\lambda}=\{x:\fhat_h(x)\ge\lambda\}.
  \label{eq:two-parameter-level-set}
\end{equation}
For fixed $h$, the sets in \eqref{eq:two-parameter-level-set} are nested as $\lambda$ decreases. This density-level ordering supplies the inclusion maps for a cluster tree. A plausible empirical tree is not itself a consistency or uncertainty statement. Such results require assumptions on geometry, separation, and estimation \citep{chaudhuri_2014_ConsistentProceduresCluster,eldridge_2015_HartiganConsistencyMerge,steinwart_2023_AdaptiveClusteringKDE}.

Raw density levels are difficult to compare across bandwidths. When a threshold exists, define the probability-content region
\begin{equation*}
  R_{h,\alpha}=\{x:f_h(x)\ge c_{h,\alpha}\},
  \qquad
  \int_{R_{h,\alpha}}f_h(x)\dd x=\alpha.
\end{equation*}
Equal content improves comparability, but it does not imply nesting in $h$. Figure~\ref{fig:levelset-heatmap} is therefore an indexed sensitivity surface rather than a bifiltration.

\begin{figure*}[t]
  \centering
  \includegraphics[width=\textwidth]{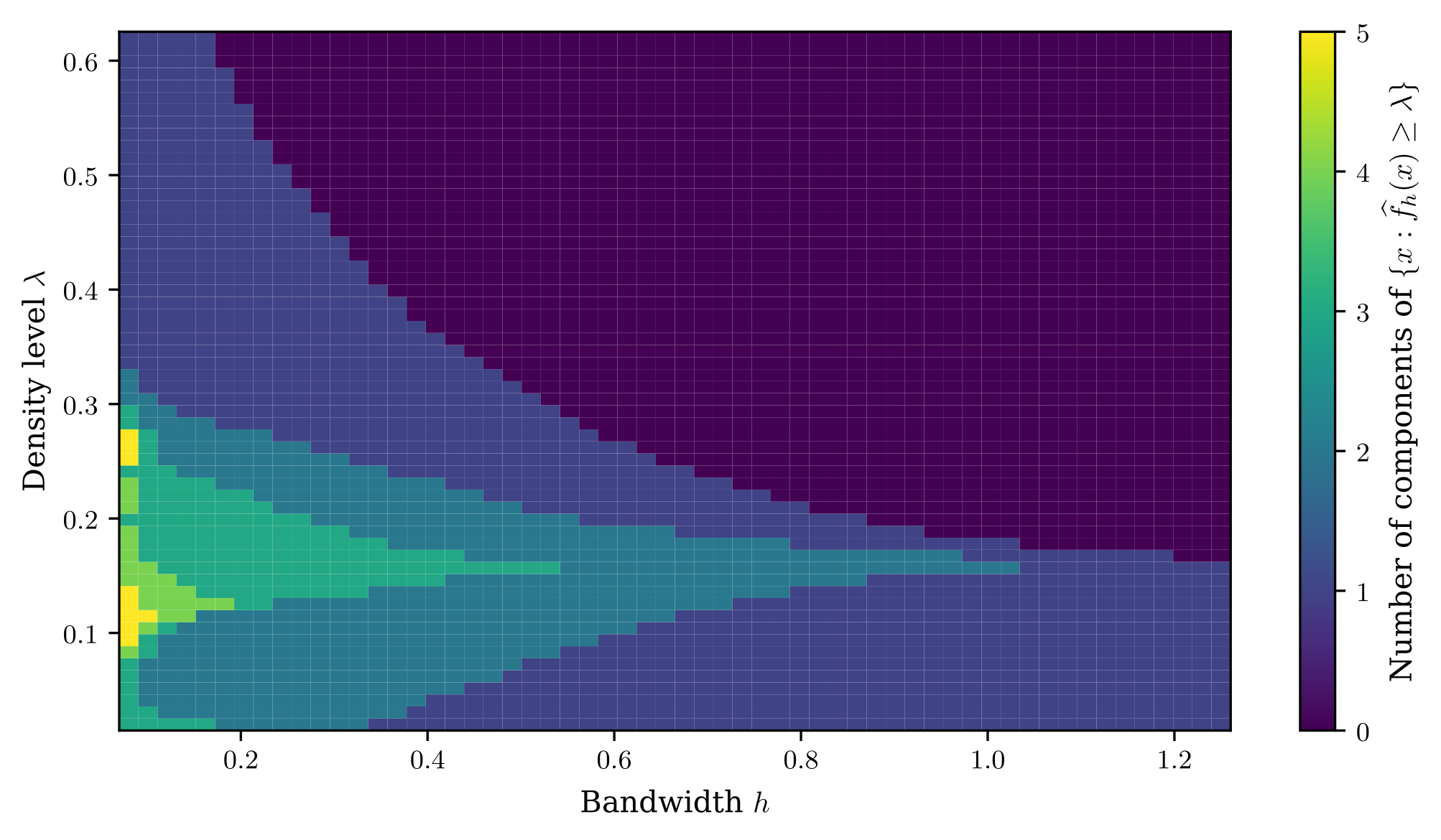}
  %\resizebox{\textwidth}{!}{\input{figures/levelset_heatmap.pgf}}
  \caption{An illustrative sensitivity surface of empirical component counts over bandwidth and density level. For each fixed bandwidth, decreasing levels form a valid filtration. The collection across bandwidths is not generally a bifiltration because cross-bandwidth inclusions can fail.}
  \label{fig:levelset-heatmap}
\end{figure*}

\subsection{Persistent homology}
\label{subsec:persistence}

Persistent homology studies topological features that are born and die along a filtration \citep{edelsbrunner_2002_TopologicalPersistenceSimplification,carlsson_2009_TopologyData,edelsbrunner_2010_ComputationalTopologyIntroduction,wasserman_2018_TopologicalDataAnalysis}. For density analysis, the superlevel sets at one fixed bandwidth form a natural filtration as $\lambda$ decreases. Zero-dimensional homology tracks connected components. Higher-dimensional homology tracks loops, voids, and related structure.

The persistence diagram records birth and death levels. A component's persistence is a difference of density levels, not the bandwidth span of a modal branch. Short persistence alone does not establish noise. A sampling threshold requires calibration for the relevant fixed-bandwidth filtration, while numerical discretization remains a separate uncertainty source.

Stability inequalities and confidence sets can support calibrated diagram thresholds \citep{cohensteiner_2007_StabilityPersistenceDiagrams,fasy_2014_ConfidenceSetsPersistence,chazal_2011_GeometricInferenceProbability,phillips_2015_GeometricInferenceKernel,chazal_2018_RobustTopologicalInference}. A threshold calibrated at one bandwidth does not control an entire bandwidth-indexed family.

\subsection{Collections across bandwidth}
\label{subsec:two-param-persistence}

Computing a cluster tree or diagram at each bandwidth produces an indexed collection. It does not automatically produce a two-parameter persistence module. Superlevel sets are nested in density level at fixed $h$, but generally not in $h$. Appendix~\ref{app:cross-bandwidth} gives an elementary Gaussian counterexample in which a fixed-level interval first expands and then contracts.

Valid summaries include fixed-bandwidth slices and adjacent bottleneck distances justified by stability. Coherently updated persistence pairings can form a vineyard \citep{cohensteiner_2006_VinesVineyards}. A generic nearest-diagram match is only a descriptive correspondence. Zigzag persistence requires explicitly defined maps \citep{carlsson_2010_ZigzagPersistence}.

Probability-content regions and component-count strips remain useful sensitivity displays. Their interpretation depends on the spatial domain, boundary treatment, connectivity, and grid resolution. These choices must be reported and varied when conclusions depend on them.

%----------------------------------------------------------------------------
\section{Inferential targets along density paths}
\label{sec:inference}

A density path estimated from data is random. Its modes, merge levels, level sets, ridges, and diagrams are therefore random. Three evidence labels prevent common misinterpretations. Descriptive evidence reports an empirical path. Inferential evidence controls an error probability for a named population target. Stability evidence reports sensitivity to resampling, preprocessing, or numerical choices.

Table~\ref{tab:method-comparison} distinguishes the claims supported by the principal scale-space tools.

\input{tables/method-comparison}

\subsection{Uniform density bands}
\label{subsec:uniform-bands}

A basic object is a confidence band for $f_h$ or for $\fhat_h-f_h$. Classical results study global deviations such as
\[
  \sup_x |\fhat_h(x)-\E\fhat_h(x)|
\]
and integrated deviations \citep{bickel_1973_GlobalMeasuresDeviations,vandervaart_1996_WeakConvergenceEmpirical,gine_2010_ConfidenceBandsDensity}. For density evolution, one wants bands over both $x$ and $h$:
\begin{equation*}
  \sup_{(x,h)\in\mathcal X\times\mathcal H}
  \left|\frac{\fhat_h(x)-\E\fhat_h(x)}{\widehat{\operatorname{se}}\{\fhat_h(x)\}}\right|.
\end{equation*}
Such bands are conceptually aligned with SiZer, but feature maps may be nonlinear and nonsmooth. Pointwise density bands do not justify scanning derivative signs across a location-bandwidth display.

Uniform bands for derivatives are particularly important. Nondegenerate modes and ridges depend on gradients and Hessians, so their analysis can require simultaneous derivative control \citep{duong_2008_FeatureSignificanceMultivariate,chacon_2013_DatadrivenDensityDerivative}. SiZer and modal tests address different targets. A modal rejection does not certify every visible branch, while nonrejection can reflect limited power near shoulders or small components.

\subsection{Bootstrap and feature stability}
\label{subsec:bootstrap-lifetimes}

Bootstrap resampling can assess whether empirical features reappear \citep{efron_1979_BootstrapMethodsAnother}. Recurrence is a stability summary unless a matching rule, estimand, calibration, and coverage result are supplied. Near a degenerate critical point, perturbations can create or remove features or switch branch identities.

Branch frequencies, pixelwise occupancy, unmatched proportions, censored endpoints, and confidence intervals are distinct outputs. An event-time distribution may mix different events or condition on survival. Similar issues occur in cluster-tree inference, ridge estimation, and topological persistence. Stable set summaries or integrated counts can be preferable to exact pointwise matching.

\subsection{Modes, ridges, and Hessian-based inference}
\label{subsec:modes-ridges-inference}

A nondegenerate interior mode has zero gradient and negative definite Hessian. Degenerate modes and shoulders require separate treatment. Nonparametric inference for modal curvature has been developed using Hessian eigenvalues \citep{genovese_2016_NonParametricInferenceDensity}. Density ridges generalize modes to lower-dimensional structures such as filaments and principal curves \citep{genovese_2014_NonparametricRidgeEstimation,chen_2015_AsymptoticTheoryDensity}.

A density-evolution version asks how evidence varies with scale. Any simultaneous statement must match its declared location-bandwidth domain. A fixed-bandwidth persistence confidence set cannot be relabeled as simultaneous evidence for a path of diagrams.

\subsection{Bayesian density evolution}
\label{subsec:bayesian-paths}

Bayesian nonparametric mixtures place a posterior distribution on the density \citep{ferguson_1973_BayesianAnalysisNonparametric,lo_1984_ClassBayesianNonparametric,escobar_1995_BayesianDensityEstimation}. A density-evolution version places a posterior on the whole path
\[
  \{f_t:t\in T\}\mid X_1,\ldots,X_n.
\]
This can be induced by smoothing posterior density draws, by evolving posterior draws through a semigroup, or by putting priors directly on stochastic flows. The posterior can then be pushed through feature maps to summarize modal counts, observed branch intervals, cluster-tree structure, or topological features.

Bayesian analysis can target a probability statement such as the posterior probability of at least three smoothed modes throughout a declared interval. The target, prior, smoothing coordinate, and conditioning must remain explicit. Posterior feature trees still require matching and summarization because branches are not canonically labeled.

The analysis should declare its transformation, spatial domain, feature definition, and finite scale grid. Data-selected ranges or features create selection questions. Sensitivity checks should vary plausible preprocessing and numerical choices, including unfavorable alternatives.

%----------------------------------------------------------------------------
\section{High-dimensional density evolution}
\label{sec:high-dimensional}

Classical KDE theory is most transparent when $d$ is fixed and $n\to\infty$. Modern data often violate this regime. Embeddings from genomics, images, language models, and neural networks can have dimension comparable to or larger than the sample size. Density evolution in high dimension must confront both statistical and geometric obstacles.

\subsection{The curse of dimensionality revisited}
\label{subsec:curse}

The AMISE rate \eqref{eq:amise-rate} shows that the optimal bandwidth scales as $n^{-1/(d+4)}$. The corresponding risk decays slowly as dimension grows. Informally, a sample size exponential in $d$ is needed to obtain the same local resolution. This is a familiar curse of dimensionality, but the density-evolution view sharpens it: at very fine scales, the path is dominated by nearest-neighbor geometry; at very coarse scales, it may be nearly Gaussian or nearly flat; useful intermediate structure may be narrow or absent.

Recent work studies KDE when $n$ and $d$ grow together with
$(\log n)/d$ fixed, identifying distinct bandwidth regimes that include a
classical central-limit regime, a heavy-tailed regime, and an
extreme-value regime in which only a few data points dominate the estimate
\citep{biroli_2026_KernelDensityEstimators}. This provides a modern explanation of a phenomenon familiar to practitioners: as bandwidth decreases in high dimension, a KDE can
transition abruptly from a smooth estimate to a collection of isolated
peaks.

\subsection{Intrinsic dimension and representation spaces}
\label{subsec:intrinsic}

High-dimensional data may lie near a low-dimensional manifold or structured subset. In such cases, density evolution in the ambient space may be misleading. Heat flow in $\R^d$ smooths perpendicular to the manifold and can obscure intrinsic geometry. Alternatives include manifold-adapted kernels, graph diffusion, local covariance adaptation, and density evolution after representation learning.

This creates a statistical dilemma. If the representation is learned from the same data, uncertainty in the representation should propagate to uncertainty in the density path. If the representation is fixed, the density path may be easier to analyze but may depend strongly on preprocessing. A useful review agenda would be to connect density evolution with intrinsic dimension estimation, manifold learning, and graph-based diffusion.

\subsection{Sliced and projected density evolution}
\label{subsec:sliced}

One pragmatic approach is to study many low-dimensional projections. For a direction $u\in S^{d-1}$, define projected data $Y_i=u^T X_i$ and a one-dimensional density path
\[
  \fhat_{h,u}(y)=\frac1n\sum_{i=1}^n K_h(y-u^T X_i).
\]
One can then aggregate modal persistence, tail behavior, or two-sample differences over random or optimized directions. Projection pursuit density estimation and sliced optimal transport suggest related strategies. The benefit is interpretability and one-dimensional scale-space theory; the cost is that high-dimensional features may be missed if they are not visible in the chosen projections.

%----------------------------------------------------------------------------
\section{Modern generative diffusion as learned density evolution}
\label{sec:generative-diffusion}

The density-evolution perspective also clarifies the connection between classical smoothing and modern generative diffusion. In score-based diffusion models, a forward stochastic process gradually transforms the data distribution into a simple prior, and a learned reverse-time process transforms the prior back into data \citep{sohl-dickstein_2015_DeepUnsupervisedLearning,song_2019_GenerativeModelingEstimating,ho_2020_DenoisingDiffusionProbabilistic,song_2020_ScoreBasedGenerativeModeling}.

Consider a forward stochastic differential equation (SDE)
\begin{equation*}
  \dd X_t=b(X_t,t)\dd t+\sigma(t)\dd W_t,
\end{equation*}
with marginal density $p_t$. The marginals solve a Fokker--Planck equation of the form \eqref{eq:fokker-planck}. Under suitable conditions, the reverse-time dynamics are again a diffusion whose drift depends on the score $\nabla\log p_t$ \citep{anderson_1982_ReversetimeDiffusionEquation,song_2020_ScoreBasedGenerativeModeling}. In a common simplified notation, the reverse drift contains the term
\begin{equation*}
  b(x,t)-\sigma(t)^2\nabla\log p_t(x).
\end{equation*}
Thus generation requires learning the score field of the evolving density.

Classical Gaussian KDE also has a score field. By \eqref{eq:gaussian-mean-shift-score}, the KDE score is proportional to a mean-shift vector. This formal connection does not make KDE a modern diffusion model, but it highlights a shared object: the gradient of the log density along a smoothing/noising path.

\subsection{Diagnostics from density evolution}
\label{subsec:generative-diagnostics}

Generative models are usually evaluated by sample quality, likelihood surrogates, feature-space distances, or downstream tasks. Density evolution suggests another class of diagnostics. Given real samples and generated samples, compute comparable density-evolution summaries in a representation space:
\[
  \{\fhat_t^{\mathrm{real}}\}_{t\ge0},
  \qquad
  \{\fhat_t^{\mathrm{gen}}\}_{t\ge0}.
\]
Then compare mode trees, cluster trees, ridge structures, topological persistence, or information-dissipation curves. Such diagnostics could detect mode dropping, spurious modes, loss of rare subpopulations, or incorrect multiscale geometry. This idea is exploratory, but it gives a concrete statistical bridge between classical density analysis and modern generative modeling.

%----------------------------------------------------------------------------
\section{Mathematical scope and limitations}
\label{sec:recommendations}

A density-evolution argument should make four structural choices explicit.
\begin{enumerate}
\item The evolution operator and its population target must be defined. A feature of $\fhat_t$ usually concerns the corresponding smoothed target $f_t$, not an unsmoothed density.
\item The index and its ordering must be specified. Bandwidth, heat time, density level, component count, and noise time have different meanings.
\item The feature map must be separated from the path. Modes, mixture components, connected regions, ridges, and homology classes are not interchangeable.
\item Any relation between indices requires explicit maps. A family indexed by two parameters is not automatically a bifiltration or a persistence module.
\end{enumerate}

Gaussian scale space has semigroup and one-dimensional causality properties. These do not extend automatically to other kernels, multivariate modal counts, or unrelated complexity parameters. Near shoulders and mergers, feature identities and global continuation remain difficult. Cross-bandwidth regions lack generic inclusion maps.

The formal results in this article are correspondingly local or conditional. The mode-motion equation applies away from degenerate critical points. The OU characterization holds within a stated class of linear Gaussian semigroups. The log-concavity certificate is sufficient rather than necessary. None of these results supplies inferential calibration for an observed feature path.

%----------------------------------------------------------------------------
\section{Open problems}
\label{sec:open-problems}

The density-evolution perspective raises several research questions with statistical and computational implications.

\smallskip 
\noindent \textbf{What should be the canonical evolution?} Heat flow is canonical for Gaussian KDE. It satisfies semigroup and smoothing properties, but it does not converge to a fitted Gaussian. The OU path in \eqref{eq:ou-gaussianization} has a natural Gaussian endpoint, but its inferential properties are largely unexplored. Mixture reduction gives a data-adaptive path, but it is algorithm-dependent.

A canonical density evolution might be required to satisfy positivity, mass preservation, affine equivariance, stability to sampling perturbations, semigroup structure, monotone information loss, and interpretable endpoints. It is unlikely that all desirable properties can hold simultaneously. Understanding the tradeoffs is an open problem.

\smallskip 
\noindent \textbf{Can multivariate scale space avoid artificial features?} One-dimensional Gaussian smoothing has a strong no-new-modes intuition. Multivariate smoothing does not have an equally simple guarantee. Can one design multivariate density evolutions for which modes, ridges, or topological features simplify monotonically? If not for all densities, can this be done for useful classes such as log-concave mixtures, elliptical mixtures, or manifold-supported distributions?

\smallskip 
\noindent \textbf{How should feature-existence intervals be inferred?} Observed branch intervals are useful but coordinate-dependent. Their endpoints can be censored by the declared range and unstable near degeneracies. Bootstrap recurrence alone does not give coverage. Valid procedures must define matching, unmatched events, selection, and the population target. Confidence regions for fixed-bandwidth persistence diagrams provide one route, while analogous theory for mode trees and ridges remains less developed.

\smallskip 
\noindent \textbf{What maps should connect topology across bandwidths?} The collection $\{L_{h,\lambda}\}$ has a valid density-level filtration for each fixed $h$, but no generic inclusions across $h$. Useful theory must supply coherent correspondences or explicit maps before invoking vineyards, zigzags, or multiparameter persistence. Statistical summaries must also separate sampling error from numerical discretization.

\smallskip 
\noindent \textbf{How should two density evolutions be compared?} Given two samples, define paths $\fhat_t^X$ and $\fhat_t^Y$. Possible distances include
\[
  \int_0^T \|\fhat_t^X-\fhat_t^Y\|_2^2\dd t,
  \qquad
  \int_0^T W_2^2(\fhat_t^X,\fhat_t^Y)\dd t,
\]
or distances between mode trees, cluster trees, or persistence diagrams. Such distances could lead to multiscale two-sample tests that identify the scales at which distributions differ.

\smallskip 
\noindent \textbf{What survives in high dimension?} In high dimension, density values are hard to estimate. It may be more realistic to estimate relative density, score fields, ridges, nearest-neighbor density ranks, or projected density evolutions. A major question is which density-evolution summaries are statistically meaningful when $d$ grows with $n$.

\smallskip 
\noindent \textbf{Can density evolution audit generative models?} Generative models can produce visually plausible samples while missing rare modes or altering multiscale structure. Density-evolution diagnostics could compare real and generated samples through mode trees, cluster trees, topological summaries, or score fields in representation space. Developing statistically calibrated versions of these diagnostics is an open problem.

\section{Conclusion}
\label{sec:conclusion}

Density estimation is often organized around estimator choice. The density-evolution perspective instead studies how explicit operators move a probability measure through a family of densities. Such a path is meaningful only after its target, index, ordering, and feature map have been fixed.

This viewpoint depends on classical density estimation rather than replacing it. Gaussian KDE is heat flow from the empirical measure. Mode trees and SiZer summarize a scale-space surface. Mixture reduction defines a separate complexity path. Cluster trees and persistent homology describe density-level filtrations, while score-based models use forward and reverse noise evolutions.

Three structural results make part of this organization precise. Nondegenerate critical points follow local differentiable branches, with velocity determined by the Hessian and the path derivative. Moment preservation and the semigroup property characterize the OU Gaussianization path within the stated linear Gaussian class. A covariance condition yields explicit terminal log-concavity bounds for shared-covariance mixtures, including heat and OU paths.

The framework also identifies limits on transfer between these constructions. A fixed-level Gaussian superlevel set need not be nested across bandwidth, and a collection of filtrations is not automatically a multiparameter persistence module. Descriptive feature duration is not inferential significance. These distinctions leave substantial open theory for global branch continuation, calibrated pathwise inference, high-dimensional structure, and generative-model diagnostics.
\begin{appendix}
\section{Proof of Theorem~\ref{thm:mode-motion}}\label{app:proof1}
\begin{proof}
Let
\[
  F(x,t)=\nabla_x f(x,t).
\]
By assumption, \(F\) is continuously differentiable in \((x,t)\). At
\((x_0,t_0)\),
\[
  F(x_0,t_0)=0,
  \qquad
  D_xF(x_0,t_0)=\nabla_x^2 f(x_0,t_0)=H_0,
\]
and \(H_0\) is nonsingular. The implicit function theorem therefore gives
neighborhoods \(V\) of \(x_0\) and \(J\) of \(t_0\), together with a unique
continuously differentiable map \(m:J\to V\), such that
\[
  m(t_0)=x_0,
  \qquad
  F(m(t),t)=0
  \quad (t\in J).
\]
If \(t_0\) is an endpoint of the interval \(I\), the same argument gives a
relative one-sided neighborhood \(J\subset I\).

Differentiating the identity \(F(m(t),t)=0\) with respect to \(t\) gives
\[
  D_xF(m(t),t)\dot m(t)+\partial_tF(m(t),t)=0.
\]
Since
\[
  D_xF(m(t),t)=\nabla_x^2 f(m(t),t),
\]
we obtain
\[
  \dot m(t)
  =
  -\{\nabla_x^2 f(m(t),t)\}^{-1}\partial_t\nabla_x f(m(t),t).
\]
If mixed derivatives commute, that is,
\[
  \partial_t\nabla_x f=\nabla_x(\partial_t f),
\]
then
\[
  \dot m(t)
  =
  -\{\nabla_x^2 f(m(t),t)\}^{-1}
   \nabla_x\{\partial_t f(m(t),t)\}.
\]
For heat flow, \(\partial_t f_t=(1/2)\Delta f_t\), so this becomes
\[
  \dot m(t)
  =
  -\frac12\{\nabla_x^2 f_t(m(t))\}^{-1}
  \nabla_x\Delta f_t(m(t)).
\]

It remains to justify the final statement concerning local constancy of
the number and Morse index of critical points. The Morse index is the number
of negative Hessian eigenvalues. Fix a compact set
\(K\subset U\) and a time \(t_0\). Suppose that no critical point of
\(f(\cdot,t_0)\) lies on \(\partial K\), and that every critical point in
\(K\) is nondegenerate. Nondegenerate critical points are isolated. Since
the critical set in \(K\) is closed and \(K\) is compact, there are only
finitely many such critical points, say \(x_1,\ldots,x_r\), in \(K\).

For each \(x_j\), the argument above gives a unique local branch
\(m_j(t)\) of critical points near \(t_0\). Choose disjoint neighborhoods
\(V_j\) of the \(x_j\)'s whose closures are contained in the interior of
\(K\). On the compact set
\[
  K\setminus \bigcup_{j=1}^r V_j,
\]
the continuous function \(\|F(x,t_0)\|\) is bounded away from zero.
By continuity of \(F\), after shrinking the time interval around \(t_0\),
there are no critical points in this complement. Hence all critical
points in \(K\) for nearby times are exactly those lying on the branches
\(m_1(t),\ldots,m_r(t)\).

Finally, the eigenvalues of the Hessian
\[
  \nabla_x^2 f(m_j(t),t)
\]
vary continuously with \(t\). Since the Hessian at \(t_0\) is nonsingular,
no eigenvalue crosses zero for \(t\) sufficiently close to \(t_0\). Thus
the Morse index, and in particular the property of being a local mode
when the Hessian is negative definite, is locally constant along each
branch. This proves the result.
\end{proof}

\section{Proof of Theorem~\ref{thm:ou-characterization}}\label{app:proof2}
\begin{proof}
Let \(P\in\mathcal P_3^+\), and write \(m=m_P\) and \(S=S_P\). By the
definition of \(T_t\),
\[
  T_tP
  =
  \mathcal L\{m+a(t)(X-m)+b(t)Z_P\},
\]
where \(X\sim P\), \(Z_P\sim \Normal(0,S)\), and \(X\) and \(Z_P\) are
independent.

First consider the covariance. Since \(X-m\) and \(Z_P\) are independent
and centered,
\[
  \Cov(T_tP)
  =
  a(t)^2 S+b(t)^2S
  =
  \{a(t)^2+b(t)^2\}S.
\]
By the assumed covariance preservation, \(\Cov(T_tP)=S\) for every
positive definite \(S\). Hence
\[
  a(t)^2+b(t)^2=1
  \qquad (t\ge0).
\]
Since \(b(t)\ge0\), this already implies
\[
  b(t)=\sqrt{1-a(t)^2}.
\]

We next use the semigroup property to identify \(a(t)\). Let
\[
  Y_t=m+a(t)(X-m)+b(t)Z_1,
\]
where \(Z_1\sim\Normal(0,S)\) is independent of \(X\). Since \(T_t\)
preserves mean and covariance, \(Y_t\) has mean \(m\) and covariance \(S\).
Applying \(T_s\) to the law of \(Y_t\) gives
\[
  m+a(s)(Y_t-m)+b(s)Z_2,
\]
where \(Z_2\sim\Normal(0,S)\) is independent of \(Y_t\). Thus
\[
  T_s(T_tP)
  =
  \mathcal L\{m+a(s)a(t)(X-m)+a(s)b(t)Z_1+b(s)Z_2\}.
\]
The Gaussian term
\[
  a(s)b(t)Z_1+b(s)Z_2
\]
is independent of \(X\) and has distribution
\[
  \Normal\bigl(0,\{a(s)^2b(t)^2+b(s)^2\}S\bigr).
\]

On the other hand,
\[
  T_{s+t}P
  =
  \mathcal L\{m+a(s+t)(X-m)+b(s+t)Z_3\},
\]
where \(Z_3\sim\Normal(0,S)\) is independent of \(X\). The semigroup
property says that these two laws are equal for every \(P\in\mathcal P_3^+\).

To identify the coefficient multiplying \(X-m\), compare centered third
moments. For a distribution \(Q\) with finite third moment, write
\[
  M_3(Q)=\E_Q\{(X-m_Q)^{\otimes 3}\}
\]
for its third central moment tensor. Centered Gaussian random variables
have zero third central moment, and all mixed third-order terms vanish
because the Gaussian noises are centered and independent of \(X\).
Therefore
\[
  M_3(T_{s+t}P)=a(s+t)^3 M_3(P),
\]
whereas
\[
  M_3(T_s(T_tP))=\{a(s)a(t)\}^3M_3(P).
\]
Because the semigroup identity holds for every \(P\in\mathcal P_3^+\), we
may choose a \(P\) with positive definite covariance and nonzero third
central moment tensor. For example, take one non-Gaussian coordinate with
nonzero third central moment and add independent Gaussian coordinates to
make the covariance positive definite. For such a \(P\),
\[
  a(s+t)^3=\{a(s)a(t)\}^3.
\]
Since \(0\le a(t)\le1\), it follows that
\[
  a(s+t)=a(s)a(t)
  \qquad (s,t\ge0).
\]

Thus \(a\) is a continuous multiplicative function on \([0,\infty)\) with
\(a(0)=1\). Moreover, \(a(t)>0\) for every \(t\ge0\). Indeed, if
\(a(t_0)=0\) for some \(t_0>0\), then
\[
  0=a(t_0)=a(t_0/2)^2,
\]
so \(a(t_0/2)=0\). Iterating gives \(a(t_0/2^k)=0\) for all \(k\), and
continuity at \(0\) would imply \(a(0)=0\), contradicting \(a(0)=1\).

Since \(a(t)>0\), define $\alpha(t)=-\log a(t).$ Then \(\alpha\) is continuous and additive:
\[
  \alpha(s+t)=\alpha(s)+\alpha(t).
\]
The continuous additive functions on \([0,\infty)\) are linear, so $
  \alpha(t)=ct$
for some \(c\ge0\). Hence
\[
  a(t)=e^{-ct}.
\]
The family is assumed not to be the identity, so \(a(t)<1\) for at least
one \(t>0\). Therefore \(c>0\). Finally,
\[
  b(t)=\sqrt{1-a(t)^2}=\sqrt{1-e^{-2ct}}.
\]
This proves the claimed characterization.
\end{proof}

\section{Proof of Theorem~\ref{thm:terminal-logconcavity}}\label{app:proof3}
\begin{proof}
Let
\[
  r_i(x)=\pi_i\varphi_\Sigma(x-\mu_i),
  \qquad
  f(x)=\sum_{i=1}^n r_i(x),
\]
so that
\[
  w_i(x)=\frac{r_i(x)}{f(x)}.
\]
For a Gaussian density with covariance \(\Sigma\),
\[
  \nabla_x \log \varphi_\Sigma(x-\mu_i)
  =
  \Sigma^{-1}(\mu_i-x).
\]
Therefore
\[
  \nabla f(x)
  =
  \sum_i r_i(x)\Sigma^{-1}(\mu_i-x)
  =
  f(x)\Sigma^{-1}\{\bar\mu_w(x)-x\},
\]
and hence
\[
  \nabla\log f(x)
  =
  \Sigma^{-1}\{\bar\mu_w(x)-x\}.
\]

We now differentiate this expression. For a direction \(v\in\R^d\),
\[
  D_v\log r_i(x)
  =
  v^T\Sigma^{-1}(\mu_i-x).
\]
Since \(w_i=r_i/f\),
\[
  D_v w_i(x)
  =
  w_i(x)
  \left[
    v^T\Sigma^{-1}(\mu_i-x)
    -
    v^T\Sigma^{-1}\{\bar\mu_w(x)-x\}
  \right],
\]
and therefore
\[
  D_v w_i(x)
  =
  w_i(x)v^T\Sigma^{-1}\{\mu_i-\bar\mu_w(x)\}.
\]
It follows that
\begin{align*}
  D_v\bar\mu_w(x)
  &=
  \sum_i \mu_i\,D_vw_i(x)  \\
  &=
  \sum_i w_i(x)\mu_i
  v^T\Sigma^{-1}\{\mu_i-\bar\mu_w(x)\}.
\end{align*}
Since
\[
  \sum_i w_i(x)\{\mu_i-\bar\mu_w(x)\}=0,
\]
we may subtract \(\bar\mu_w(x)\) from \(\mu_i\) inside the preceding sum,
which gives
\[
  D_v\bar\mu_w(x)
  =
  C_w(x)\Sigma^{-1}v.
\]
Thus the Jacobian of \(x\mapsto\bar\mu_w(x)\) is
\[
  D\bar\mu_w(x)=C_w(x)\Sigma^{-1}.
\]
Differentiating
\[
  \nabla\log f(x)=\Sigma^{-1}\{\bar\mu_w(x)-x\}
\]
therefore yields
\[
  \nabla^2\log f(x)
  =
  \Sigma^{-1}C_w(x)\Sigma^{-1}-\Sigma^{-1}.
\]
This is the stated Hessian identity.

If \(C_w(x)\preceq \Sigma\) for every \(x\), then
\[
  \nabla^2\log f(x)
  =
  \Sigma^{-1}\{C_w(x)-\Sigma\}\Sigma^{-1}
  \preceq 0
\]
for every \(x\), and hence \(\log f\) is concave. Thus \(f\) is
log-concave. If \(C_w(x)\prec\Sigma\) for every \(x\), the Hessian is
negative definite for every \(x\), and \(f\) is strictly log-concave.

It remains to prove the ellipsoid condition. Suppose that for some
\(c\in\R^d\),
\[
  \rho^2:=\max_i\|\Sigma^{-1/2}(\mu_i-c)\|^2<1.
\]
Let \(v\neq0\). Then
\[
  v^TC_w(x)v
  =
  \sum_i w_i(x)
  \{v^T(\mu_i-\bar\mu_w(x))\}^2
  =
  \Var_w\{v^T\mu_i\}.
\]
Using the elementary inequality
\[
  \Var_w\{v^T\mu_i\}
  \le
  \E_w\{v^T(\mu_i-c)\}^2,
\]
we obtain
\[
  v^TC_w(x)v
  \le
  \sum_i w_i(x)\{v^T(\mu_i-c)\}^2.
\]
Write
\[
  y_i=\Sigma^{-1/2}(\mu_i-c),
  \qquad
  q=\Sigma^{1/2}v.
\]
Then
\[
  v^T(\mu_i-c)=q^Ty_i,
\]
so by Cauchy--Schwarz,
\[
  \{v^T(\mu_i-c)\}^2
  =
  (q^Ty_i)^2
  \le
  \|q\|^2\|y_i\|^2
  \le
  \rho^2\|q\|^2.
\]
Therefore
\[
  v^TC_w(x)v
  \le
  \rho^2\|q\|^2
  =
  \rho^2 v^T\Sigma v
  <
  v^T\Sigma v.
\]
Since this holds for every \(v\neq0\), \(C_w(x)\prec\Sigma\) for every
\(x\). Hence \(f\) is strictly log-concave. This criterion is now applied to the two density evolutions in the theorem.

First consider the heat-flow KDE
\[
  \widehat u_t(x)=\frac1n\sum_{i=1}^n\varphi_{tI_d}(x-X_i).
\]
Here the common covariance is \(\Sigma=tI_d\), the component means are
\(\mu_i=X_i\), and the weights are equal. If
\[
  \max_i\|X_i-c\|^2\le R^2,
\]
then
\[
  \max_i\|\Sigma^{-1/2}(\mu_i-c)\|^2
  =
  \max_i\|(tI_d)^{-1/2}(X_i-c)\|^2
  \le
  \frac{R^2}{t}.
\]
Thus the strict ellipsoid condition holds whenever \(t>R^2\). Hence
\(\widehat u_t\) is strictly log-concave for all \(t>R^2\).

Next consider the OU path
\[
  g_t(x)
  =
  \frac1n\sum_{i=1}^n
  \Normal\left(
    x\mid \bar X+e^{-t}(X_i-\bar X),
    (1-e^{-2t})S
  \right).
\]
For \(t>0\), this is a shared-covariance Gaussian mixture with
\[
  \Sigma_t=(1-e^{-2t})S,
  \qquad
  \mu_i(t)=\bar X+e^{-t}(X_i-\bar X).
\]
Choose \(c=\bar X\). Define $d_i = X_i - \bar{X}$, then 
\begin{align*}
  \|\Sigma_t^{-1/2}\{\mu_i(t)-\bar X\}\|^2
  &=
  e^{-2t}d_i^T\{(1-e^{-2t})S\}^{-1}d_i\\
  &=
  \frac{e^{-2t}}{1-e^{-2t}}
  d_i^TS^{-1}d_i.
\end{align*}
Therefore the ellipsoid condition is
\[
  \frac{e^{-2t}}{1-e^{-2t}}R_S^2<1,
\]
where
\[
  R_S^2=\max_i~ d_i^T S^{-1} d_i.
\]
This inequality is equivalent to
\[
  e^{-2t}(1+R_S^2)<1,
\]
or
\[
  t>\frac12\log(1+R_S^2).
\]
Hence \(g_t\) is strictly log-concave whenever this condition holds.

Finally, a strictly log-concave continuous density on \(\R^d\) that tends
to zero at infinity has a unique global maximizer. Indeed, if two distinct
points were both global maximizers, strict concavity of \(\log f\) would
make every point in the open line segment between them have strictly larger
log density, a contradiction. Its positive-level superlevel sets are convex because, for every
\(\lambda>0\),
\[
  \{x:f(x)\ge \lambda\}
  =
  \{x:\log f(x)\ge \log \lambda\}
\]
is a superlevel set of a concave function. The case \(\lambda\le0\) is
trivial for a positive density. This
proves the theorem.
\end{proof}

\clearpage
\section{An elementary failure of cross-bandwidth nesting}
\label{app:cross-bandwidth}

Let $P$ be a point mass at zero on the real line. Its Gaussian-smoothed density is
\[
  f_h(x)=\frac{1}{\sqrt{2\pi}h}
  \exp\left\{-\frac{x^2}{2h^2}\right\}.
\]
For fixed $\lambda>0$, set $H=(\sqrt{2\pi}\lambda)^{-1}$. When $0<h<H$, the superlevel set is $[-r(h),r(h)]$, where
\[
  r(h)^2=2h^2\log(H/h),
  \qquad
  \frac{\dd}{\dd h}r(h)^2=2h\{2\log(H/h)-1\}.
\]
The radius increases until $h=H/\sqrt{\mathrm e}$, then decreases to zero at $h=H$. For $h>H$, the set is empty. Thus even this fixed-level Gaussian superlevel set is not monotone in bandwidth. The example demonstrates the missing inclusion property. It does not claim that every region path has the same behavior.

For probability content $\alpha$, the thresholded region satisfies
\[
  R_{h,\alpha}=\{x:f_h(x)\geq c_{h,\alpha}\},
  \qquad
  \int_{R_{h,\alpha}}f_h(x)\dd x=\alpha,
\]
when such a threshold exists. Equal content improves cross-bandwidth comparison but does not generally provide filtration maps. Level-set, cluster-tree, and persistence procedures must retain their own targets and perturbation metrics \citep{polonik_1995_MeasuringMassConcentrations,tsybakov_1997_NonparametricEstimationDensity,rigollet_2009_OptimalRatesPlugin,chaudhuri_2010_RatesConvergenceCluster}.

\clearpage
\section{Review scope}
\label{app:review-scope}

This is a scoped narrative review. Source discovery began with foundational references and used targeted searches for kernel estimation, Gaussian scale space, critical bandwidths, modality tests, mode trees, mode forests, SiZer, mixtures, density level sets, cluster trees, persistence, diffusion, and score-based generation. Searches were updated on September 2, 2026. Inclusion required a direct role in defining a target, supporting a claim, specifying a method, or documenting a failure mode used in the article.

\end{appendix}

%%  MR numbers will be added by VTeX.                      %%
%%                                                         %%
%%  Use \cite{...} to cite references in text.             %%
%%                                                         %%
%%%%%%%%%%%%%%%%%%%%%%%%%%%%%%%%%%%%%%%%%%%%%%%%%%%%%%%%%%%%%

%% if your bibliography is in bibtex format, uncomment commands:
\bibliographystyle{imsart-number} % Style BST file (imsart-number.bst or imsart-nameyear.bst)
\bibliography{references}       % Bibliography file (usually '*.bib')

\end{document}

%% file: tables/method-comparison.tex
\begin{table}[!htbp]
\centering
\caption{Targets, evidence, and limitations of scale-space summaries.}
\label{tab:method-comparison}
\begin{tabularx}{\textwidth}{@{}>{\raggedright\arraybackslash}p{0.23\textwidth}>{\raggedright\arraybackslash}p{0.25\textwidth}>{\raggedright\arraybackslash}X@{}}
\toprule
Summary & Target and evidence & Supported claim and limitation \\
\midrule
Mode tree & Estimated modes over bandwidth. Descriptive. & A branch exists over the observed range. Its length is not significance, and boundary endpoints are censored. \\
Mode forest & Bootstrap modal occupancy in location--bandwidth pixels. Stability. & A mode falls in a pixel with the reported frequency. Complete branches are not matched, and no coverage is asserted. \\
Critical bandwidth and modal test & Empirical modal restriction or a calibrated population null. & The critical bandwidth is descriptive until calibrated. The ACR test concerns exactly $k$ versus more modes. Nonrejection gives neither a count nor an upper bound. \\
SiZer & Smoothed derivative signs. Finite-grid inference. & Jointly calibrated signs support increasing or decreasing regions. Unsmoothed targets, continua, and data-selected ranges need further justification. \\
Cluster tree or content region & Connected upper-level sets at fixed bandwidth. & Components are nested in density level, not generally in bandwidth. Inference requires an additional procedure. \\
Calibrated zero-dimensional persistence & Component lifetimes along a fixed-bandwidth filtration. & Lifetimes exceed a sampling threshold on the declared numerical grid. This does not cover the entire bandwidth path or discretization error. \\
\bottomrule
\end{tabularx}
\par{\raggedright\noindent Notes: ACR denotes the Ameijeiras-Alonso--Crujeiras--Rodr\'iguez-Casal modality test. A claim is conditional on the selected procedure's assumptions.\par}
\end{table}

%% file: references.bib
@string{jan = "January"}

@string{mar = "March"}

@string{apr = "April"}

@string{may = "May"}

@string{jun = "June"}

@string{jul = "July"}

@string{aug = "August"}

@string{sep = "September"}

@string{oct = "October"}

@string{nov = "November"}

@string{dec = "December"}

@inproceedings{witkin_1983_ScalespaceFiltering,
	address = {Los Altos, CA},
	series = {{IJCAI}'83},
	title = {Scale-Space Filtering},
	booktitle = {Proceedings of the Eighth International Joint Conference on Artificial Intelligence},
	volume = {2},
	editor = {Bundy, Alan},
	publisher = {William Kaufmann},
	author = {Witkin, Andrew P.},
	year = {1983},
	pages = {1019--1022},
	url = {https://www.ijcai.org/Proceedings/83-2/Papers/091.pdf},
}

@article{wasserman_2018_TopologicalDataAnalysis,
	title = {Topological {Data} {Analysis}},
	volume = {5},
	issn = {2326-8298, 2326-831X},
	doi = {10.1146/annurev-statistics-031017-100045},
	language = {en},
	number = {1},
	urldate = {2026-05-28},
	journal = {Annual Review of Statistics and Its Application},
	author = {Wasserman, Larry},
	month = mar,
	year = {2018},
	pages = {501--532},
}

@book{wand_1994_KernelSmoothing,
	address = {London},
	title = {Kernel {Smoothing}},
	isbn = {978-0-429-17059-1},
	doi = {10.1201/b14876},
	language = {en},
	urldate = {2026-05-28},
	publisher = {Chapman \& Hall},
	author = {Wand, M. P. and Jones, M. C.},
	year = {1995},
}

@article{tsybakov_1997_NonparametricEstimationDensity,
	title = {On Nonparametric Estimation of Density Level Sets},
	pages = {948--969},
	volume = {25},
	issn = {0090-5364},
	doi = {10.1214/aos/1069362732},
	number = {3},
	urldate = {2026-05-28},
	journal = {The Annals of Statistics},
	author = {Tsybakov, A. B.},
	month = jun,
	year = {1997},
}

@article{terrell_1992_VariableKernelDensity,
	title = {Variable {Kernel} {Density} {Estimation}},
	volume = {20},
	issn = {0090-5364},
	doi = {10.1214/aos/1176348768},
	number = {3},
	urldate = {2026-05-28},
	journal = {The Annals of Statistics},
	author = {Terrell, George R. and Scott, David W.},
	month = sep,
	year = {1992},
}

@article{stuetzle_2010_GeneralizedSingleLinkage,
	title = {A {Generalized} {Single} {Linkage} {Method} for {Estimating} the {Cluster} {Tree} of a {Density}},
	volume = {19},
	issn = {1061-8600, 1537-2715},
	doi = {10.1198/jcgs.2009.07049},
	language = {en},
	number = {2},
	urldate = {2026-05-28},
	journal = {Journal of Computational and Graphical Statistics},
	author = {Stuetzle, Werner and Nugent, Rebecca},
	month = jan,
	year = {2010},
	pages = {397--418},
}

@article{stuetzle_2003_EstimatingClusterTree,
	title = {Estimating the {Cluster} {Tree} of a {Density} by {Analyzing} the {Minimal} {Spanning} {Tree} of a {Sample}},
	volume = {20},
	copyright = {http://www.springer.com/tdm},
	issn = {0176-4268, 1432-1343},
	doi = {10.1007/s00357-003-0004-6},
	number = {1},
	urldate = {2026-05-28},
	journal = {Journal of Classification},
	author = {Stuetzle, Werner},
	month = may,
	year = {2003},
	pages = {25--47},
}

@article{steinwart_2015_FullyAdaptiveDensitybased,
	title = {Fully adaptive density-based clustering},
	volume = {43},
	issn = {0090-5364},
	url = {https://projecteuclid.org/journals/annals-of-statistics/volume-43/issue-5/Fully-adaptive-density-based-clustering/10.1214/15-AOS1331.full},
	doi = {10.1214/15-AOS1331},
	number = {5},
	urldate = {2026-05-28},
	journal = {The Annals of Statistics},
	author = {Steinwart, Ingo},
	month = oct,
	year = {2015},
}

@article{stam_1959_InequalitiesSatisfiedQuantities,
	title = {Some inequalities satisfied by the quantities of information of {Fisher} and {Shannon}},
	volume = {2},
	issn = {00199958},
	doi = {10.1016/S0019-9958(59)90348-1},
	language = {en},
	number = {2},
	urldate = {2026-05-28},
	journal = {Information and Control},
	author = {Stam, A.J.},
	month = jun,
	year = {1959},
	pages = {101--112},
}

@misc{song_2020_ScoreBasedGenerativeModeling,
	title = {Score-{Based} {Generative} {Modeling} through {Stochastic} {Differential} {Equations}},
	copyright = {arXiv.org perpetual, non-exclusive license},
	url = {https://arxiv.org/abs/2011.13456},
	doi = {10.48550/ARXIV.2011.13456},
	urldate = {2026-05-28},
	publisher = {arXiv},
	author = {Song, Yang and Sohl-Dickstein, Jascha and Kingma, Diederik P. and Kumar, Abhishek and Ermon, Stefano and Poole, Ben},
	year = {2020},
}

@incollection{song_2019_GenerativeModelingEstimating,
	address = {Red Hook, NY, USA},
	title = {Generative modeling by estimating gradients of the data distribution},
	booktitle = {Proceedings of the 33rd {International} {Conference} on {Neural} {Information} {Processing} {Systems}},
	publisher = {Curran Associates Inc.},
	author = {Song, Yang and Ermon, Stefano},
	year = {2019},
}

@inproceedings{sohl-dickstein_2015_DeepUnsupervisedLearning,
	address = {Lille, France},
	series = {Proceedings of {Machine} {Learning} {Research}},
	title = {Deep {Unsupervised} {Learning} using {Nonequilibrium} {Thermodynamics}},
	volume = {37},
	url = {https://proceedings.mlr.press/v37/sohl-dickstein15.html},
	booktitle = {Proceedings of the 32nd {International} {Conference} on {Machine} {Learning}},
	publisher = {PMLR},
	author = {Sohl-Dickstein, Jascha and Weiss, Eric and Maheswaranathan, Niru and Ganguli, Surya},
	editor = {Bach, Francis and Blei, David},
	month = jul,
	year = {2015},
	pages = {2256--2265},
}

@article{silverman_1981_UsingKernelDensity,
	title = {Using {Kernel} {Density} {Estimates} to {Investigate} {Multimodality}},
	volume = {43},
	issn = {1369-7412, 1467-9868},
	doi = {10.1111/j.2517-6161.1981.tb01155.x},
	language = {en},
	number = {1},
	urldate = {2026-05-28},
	journal = {Journal of the Royal Statistical Society Series B: Statistical Methodology},
	author = {Silverman, B. W.},
	month = sep,
	year = {1981},
	pages = {97--99},
}

@article{scott_2001_KernelsMixtures,
	title = {From {Kernels} to {Mixtures}},
	volume = {43},
	issn = {0040-1706, 1537-2723},
	url = {http://www.tandfonline.com/doi/abs/10.1198/004017001316975916},
	doi = {10.1198/004017001316975916},
	language = {en},
	number = {3},
	urldate = {2026-05-28},
	journal = {Technometrics},
	author = {Scott, David W and Szewczyk, William F},
	month = aug,
	year = {2001},
	pages = {323--335},
}

@book{scott_2015_MultivariateDensityEstimation,
	edition = {2nd},
	address = {Hoboken, NJ},
	series = {Wiley {Series} in {Probability} and {Statistics}},
	title = {Multivariate {Density} {Estimation}: {Theory}, {Practice}, and {Visualization}},
	isbn = {978-0-471-69755-8 978-1-118-57557-4},
	shorttitle = {Multivariate {Density} {Estimation}},
	doi = {10.1002/9781118575574},
	language = {en},
	urldate = {2026-05-28},
	publisher = {Wiley},
	author = {Scott, David W.},
	month = mar,
	year = {2015},
}

@article{schoenberg_1951_PolyaFrequencyFunctions,
	title = {On {Polya} Frequency Functions: {I}. The Totally Positive Functions and Their {Laplace} Transforms},
	volume = {1},
	issn = {0021-7670, 1565-8538},
	shorttitle = {On {Polya} frequency functions},
	doi = {10.1007/BF02790092},
	language = {en},
	number = {1},
	urldate = {2026-05-28},
	journal = {Journal d'Analyse Mathématique},
	author = {Schoenberg, I. J.},
	month = dec,
	year = {1951},
	pages = {331--374},
}

@article{samworth_2010_AsymptoticsOptimalBandwidth,
	title = {Asymptotics and optimal bandwidth selection for highest density region estimation},
	volume = {38},
	issn = {0090-5364},
	url = {https://projecteuclid.org/journals/annals-of-statistics/volume-38/issue-3/Asymptotics-and-optimal-bandwidth-selection-for-highest-density-region-estimation/10.1214/09-AOS766.full},
	doi = {10.1214/09-AOS766},
	number = {3},
	urldate = {2026-05-28},
	journal = {The Annals of Statistics},
	author = {Samworth, R. J. and Wand, M. P.},
	month = jun,
	year = {2010},
}

@article{rosenblatt_1956_RemarksNonparametricEstimates,
	title = {Remarks on {Some} {Nonparametric} {Estimates} of a {Density} {Function}},
	volume = {27},
	issn = {0003-4851},
	url = {http://projecteuclid.org/euclid.aoms/1177728190},
	doi = {10.1214/aoms/1177728190},
	language = {en},
	number = {3},
	urldate = {2026-05-28},
	journal = {The Annals of Mathematical Statistics},
	author = {Rosenblatt, Murray},
	month = sep,
	year = {1956},
	pages = {832--837},
}

@article{rinaldo_2010_GeneralizedDensityClustering,
	title = {Generalized Density Clustering},
	pages = {2678--2722},
	volume = {38},
	issn = {0090-5364},
	url = {https://projecteuclid.org/journals/annals-of-statistics/volume-38/issue-5/Generalized-density-clustering/10.1214/10-AOS797.full},
	doi = {10.1214/10-AOS797},
	number = {5},
	urldate = {2026-05-28},
	journal = {The Annals of Statistics},
	author = {Rinaldo, Alessandro and Wasserman, Larry},
	month = oct,
	year = {2010},
}

@article{rigollet_2009_OptimalRatesPlugin,
	title = {Optimal Rates for Plug-In Estimators of Density Level Sets},
	pages = {1154--1178},
	volume = {15},
	issn = {1350-7265},
	url = {https://projecteuclid.org/journals/bernoulli/volume-15/issue-4/Optimal-rates-for-plug-in-estimators-of-density-level-sets/10.3150/09-BEJ184.full},
	doi = {10.3150/09-BEJ184},
	number = {4},
	urldate = {2026-05-28},
	journal = {Bernoulli},
	author = {Rigollet, Philippe and Vert, R{\'e}gis},
	month = nov,
	year = {2009},
}

@article{richardson_1997_BayesianAnalysisMixtures,
	title = {On {Bayesian} {Analysis} of {Mixtures} with an {Unknown} {Number} of {Components} (with discussion)},
	volume = {59},
	issn = {1369-7412, 1467-9868},
	doi = {10.1111/1467-9868.00095},
	language = {en},
	number = {4},
	urldate = {2026-05-28},
	journal = {Journal of the Royal Statistical Society Series B: Statistical Methodology},
	author = {Richardson, Sylvia. and Green, Peter J.},
	month = nov,
	year = {1997},
	pages = {731--792},
}

@article{ray_2005_TopographyMultivariateNormal,
	title = {The Topography of Multivariate Normal Mixtures},
	pages = {2042--2065},
	volume = {33},
	issn = {0090-5364},
	url = {https://projecteuclid.org/journals/annals-of-statistics/volume-33/issue-5/The-topography-of-multivariate-normal-mixtures/10.1214/009053605000000417.full},
	doi = {10.1214/009053605000000417},
	number = {5},
	urldate = {2026-05-28},
	journal = {The Annals of Statistics},
	author = {Ray, Surajit and Lindsay, Bruce G.},
	month = oct,
	year = {2005},
}

@article{priebe_2000_AlternatingKernelMixture,
	title = {Alternating kernel and mixture density estimates},
	volume = {35},
	issn = {01679473},
	doi = {10.1016/S0167-9473(00)00003-7},
	language = {en},
	number = {1},
	urldate = {2026-05-28},
	journal = {Computational Statistics \& Data Analysis},
	author = {Priebe, Carey E. and Marchette, David J.},
	month = nov,
	year = {2000},
	pages = {43--65},
}

@article{polonik_1995_MeasuringMassConcentrations,
	title = {Measuring Mass Concentrations and Estimating Density Contour Clusters---An Excess Mass Approach},
	pages = {855--881},
	volume = {23},
	issn = {0090-5364},
	doi = {10.1214/aos/1176324626},
	number = {3},
	urldate = {2026-05-28},
	journal = {The Annals of Statistics},
	author = {Polonik, Wolfgang},
	month = jun,
	year = {1995},
}

@inproceedings{phillips_2015_GeometricInferenceKernel,
	series = {{LIPIcs}},
	title = {Geometric {Inference} on {Kernel} {Density} {Estimates}},
	doi = {10.4230/LIPICS.SOCG.2015.857},
	booktitle = {31st {International} {Symposium} on {Computational} {Geometry}, {SoCG} 2015, {Eindhoven}, {The} {Netherlands}, {June} 22-25, 2015},
	publisher = {Schloss Dagstuhl - Leibniz-Zentrum für Informatik},
	author = {Phillips, Jeff M. and Wang, Bei and Zheng, Yan},
	editor = {Arge, Lars and Pach, János},
	year = {2015},
	pages = {857--871},
}

@article{parzen_1962_EstimationProbabilityDensity,
	title = {On {Estimation} of a {Probability} {Density} {Function} and {Mode}},
	volume = {33},
	issn = {0003-4851},
	doi = {10.1214/aoms/1177704472},
	language = {en},
	number = {3},
	urldate = {2026-05-28},
	journal = {The Annals of Mathematical Statistics},
	author = {Parzen, Emanuel},
	month = sep,
	year = {1962},
	pages = {1065--1076},
}

@article{minnotte_1993_ModeTreeTool,
	title = {The {Mode} {Tree}: {A} {Tool} for {Visualization} of {Nonparametric} {Density} {Features}},
	volume = {2},
	issn = {1061-8600, 1537-2715},
	shorttitle = {The {Mode} {Tree}},
	doi = {10.1080/10618600.1993.10474599},
	language = {en},
	number = {1},
	urldate = {2026-05-28},
	journal = {Journal of Computational and Graphical Statistics},
	author = {Minnotte, Michael C. and Scott, David W.},
	month = mar,
	year = {1993},
	pages = {51--68},
}

@article{minnotte_1998_BumpyRoadModeForest,
	title = {The {Bumpy} {Road} to the {Mode} {Forest}},
	volume = {7},
	number = {2},
	journal = {Journal of Computational and Graphical Statistics},
	author = {Minnotte, Michael C. and Marchette, David J. and Wegman, Edward J.},
	year = {1998},
	pages = {239--251},
	doi = {10.1080/10618600.1998.10474773},
}

@article{minnotte_1997_NonparametricTestingExistence,
	title = {Nonparametric Testing of the Existence of Modes},
	pages = {1646--1660},
	volume = {25},
	issn = {0090-5364},
	doi = {10.1214/aos/1031594735},
	number = {4},
	urldate = {2026-05-28},
	journal = {The Annals of Statistics},
	author = {Minnotte, Michael C.},
	month = aug,
	year = {1997},
}

@book{lindsay_1995_MixtureModelsTheory,
	address = {Hayward, Calif. : Alexandria, Va},
	series = {{NSF}-{CBMS} regional conference series in probability and statistics},
	title = {Mixture models: theory, geometry, and applications},
	isbn = {978-0-940600-32-4},
	shorttitle = {Mixture models},
	number = {v. 5},
	publisher = {Institute of Mathematical Statistics ; American Statistical Association},
	author = {Lindsay, Bruce G.},
	year = {1995},
}

@article{lo_1984_ClassBayesianNonparametric,
	title = {On a {Class} of {Bayesian} {Nonparametric} {Estimates}: {I}. {Density} {Estimates}},
	volume = {12},
	issn = {0090-5364},
	shorttitle = {On a {Class} of {Bayesian} {Nonparametric} {Estimates}},
	doi = {10.1214/aos/1176346412},
	number = {1},
	urldate = {2026-05-28},
	journal = {The Annals of Statistics},
	author = {Lo, Albert Y.},
	month = mar,
	year = {1984},
}

@article{lindeberg_1994_ScalespaceTheoryBasic,
	title = {Scale-Space Theory: A Basic Tool for Analyzing Structures at Different Scales},
	volume = {21},
	issn = {0266-4763, 1360-0532},
	shorttitle = {Scale-space theory},
	doi = {10.1080/757582976},
	language = {en},
	number = {1-2},
	urldate = {2026-05-28},
	journal = {Journal of Applied Statistics},
	author = {Lindeberg, Tony},
	month = jan,
	year = {1994},
	pages = {225--270},
}

@article{koenderink_1984_StructureImages,
	title = {The Structure of Images},
	volume = {50},
	copyright = {http://www.springer.com/tdm},
	issn = {0340-1200, 1432-0770},
	url = {http://link.springer.com/10.1007/BF00336961},
	doi = {10.1007/BF00336961},
	language = {en},
	number = {5},
	urldate = {2026-05-28},
	journal = {Biological Cybernetics},
	author = {Koenderink, Jan J.},
	month = aug,
	year = {1984},
	pages = {363--370},
}

@article{kiefer_1956_ConsistencyMaximumLikelihood,
	title = {Consistency of the {Maximum} {Likelihood} {Estimator} in the {Presence} of {Infinitely} {Many} {Incidental} {Parameters}},
	volume = {27},
	issn = {0003-4851},
	doi = {10.1214/aoms/1177728066},
	language = {en},
	number = {4},
	urldate = {2026-05-28},
	journal = {The Annals of Mathematical Statistics},
	author = {Kiefer, J. and Wolfowitz, J.},
	month = dec,
	year = {1956},
	pages = {887--906},
}

@article{jordan_1998_VariationalFormulationFokkerPlanck,
	title = {The {Variational} {Formulation} of the {Fokker}--{Planck} {Equation}},
	volume = {29},
	issn = {0036-1410, 1095-7154},
	doi = {10.1137/S0036141096303359},
	language = {en},
	number = {1},
	urldate = {2026-05-28},
	journal = {SIAM Journal on Mathematical Analysis},
	author = {Jordan, Richard and Kinderlehrer, David and Otto, Felix},
	month = jan,
	year = {1998},
	pages = {1--17},
}

@inproceedings{ho_2020_DenoisingDiffusionProbabilistic,
	title = {Denoising {Diffusion} {Probabilistic} {Models}},
	volume = {33},
	url = {https://proceedings.neurips.cc/paper_files/paper/2020/file/4c5bcfec8584af0d967f1ab10179ca4b-Paper.pdf},
	booktitle = {Advances in {Neural} {Information} {Processing} {Systems}},
	publisher = {Curran Associates, Inc.},
	author = {Ho, Jonathan and Jain, Ajay and Abbeel, Pieter},
	editor = {Larochelle, H. and Ranzato, M. and Hadsell, R. and Balcan, M.F. and Lin, H.},
	year = {2020},
	pages = {6840--6851},
}

@book{hartigan_1975_ClusteringAlgorithms,
	address = {New York},
	series = {Wiley series in probability and mathematical statistics},
	title = {Clustering Algorithms},
	isbn = {978-0-471-35645-5},
	publisher = {Wiley},
	author = {Hartigan, John A.},
	year = {1975},
}

@article{hall_2001_CalibrationSilvermansTest,
	title = {On the {Calibration} of {Silverman}'s {Test} for {Multimodality}},
	volume = {11},
	issn = {1017-0405, 1996-8507},
	url = {https://www3.stat.sinica.edu.tw/statistica/j11n2/j11n28/j11n28.htm},
	number = {2},
	urldate = {2026-05-28},
	journal = {Statistica Sinica},
	publisher = {Institute of Statistical Science, Academia Sinica},
	author = {Hall, Peter and York, Matthew},
	month = apr,
	year = {2001},
	pages = {515--536},
}

@article{gine_2010_ConfidenceBandsDensity,
	title = {Confidence bands in density estimation},
	volume = {38},
	issn = {0090-5364},
	doi = {10.1214/09-AOS738},
	number = {2},
	urldate = {2026-05-28},
	journal = {The Annals of Statistics},
	author = {Giné, Evarist and Nickl, Richard},
	month = apr,
	year = {2010},
}

@article{genovese_2016_NonParametricInferenceDensity,
	title = {Non-{Parametric} {Inference} for {Density} {Modes}},
	volume = {78},
	issn = {1369-7412, 1467-9868},
	doi = {10.1111/rssb.12111},
	language = {en},
	number = {1},
	urldate = {2026-05-28},
	journal = {Journal of the Royal Statistical Society Series B: Statistical Methodology},
	author = {Genovese, Christopher R. and Perone-Pacifico, Marco and Verdinelli, Isabella and Wasserman, Larry},
	month = jan,
	year = {2016},
	pages = {99--126},
}

@article{genovese_2014_NonparametricRidgeEstimation,
	title = {Nonparametric ridge estimation},
	volume = {42},
	issn = {0090-5364},
	doi = {10.1214/14-AOS1218},
	number = {4},
	urldate = {2026-05-28},
	journal = {The Annals of Statistics},
	author = {Genovese, Christopher R. and Perone-Pacifico, Marco and Verdinelli, Isabella and Wasserman, Larry},
	month = aug,
	year = {2014},
}

@article{fukunaga_1975_EstimationGradientDensity,
	title = {The Estimation of the Gradient of a Density Function, With Applications in Pattern Recognition},
	volume = {21},
	issn = {0018-9448, 1557-9654},
	doi = {10.1109/TIT.1975.1055330},
	number = {1},
	urldate = {2026-05-28},
	journal = {IEEE Transactions on Information Theory},
	author = {Fukunaga, K. and Hostetler, L.},
	month = jan,
	year = {1975},
	pages = {32--40},
}

@article{fraley_2002_ModelBasedClusteringDiscriminant,
	title = {Model-{Based} {Clustering}, {Discriminant} {Analysis}, and {Density} {Estimation}},
	volume = {97},
	issn = {0162-1459, 1537-274X},
	doi = {10.1198/016214502760047131},
	language = {en},
	number = {458},
	urldate = {2026-05-28},
	journal = {Journal of the American Statistical Association},
	author = {Fraley, Chris and Raftery, Adrian E},
	month = jun,
	year = {2002},
	pages = {611--631},
}

@article{fisher_2001_ModeTestingExcess,
	title = {Mode Testing via the Excess Mass Estimate},
	volume = {88},
	issn = {0006-3444, 1464-3510},
	doi = {10.1093/biomet/88.2.499},
	language = {en},
	number = {2},
	urldate = {2026-05-28},
	journal = {Biometrika},
	author = {Fisher, N. I. and Marron, J. S.},
	month = jun,
	year = {2001},
	pages = {499--517},
}

@article{figueiredo_2002_UnsupervisedLearningFinite,
	title = {Unsupervised {Learning} of {Finite} {Mixture} {Models}},
	volume = {24},
	issn = {0162-8828, 2160-9292},
	doi = {10.1109/34.990138},
	number = {3},
	urldate = {2026-05-28},
	journal = {IEEE Transactions on Pattern Analysis and Machine Intelligence},
	author = {Figueiredo, M.A.T. and Jain, A.K.},
	month = mar,
	year = {2002},
	pages = {381--396},
}

@article{ferguson_1973_BayesianAnalysisNonparametric,
	title = {A {Bayesian} {Analysis} of {Some} {Nonparametric} {Problems}},
	volume = {1},
	issn = {0090-5364},
	doi = {10.1214/aos/1176342360},
	number = {2},
	urldate = {2026-05-28},
	journal = {The Annals of Statistics},
	author = {Ferguson, Thomas S.},
	month = mar,
	year = {1973},
}

@article{fasy_2014_ConfidenceSetsPersistence,
	title = {Confidence Sets for Persistence Diagrams},
	pages = {2301--2339},
	volume = {42},
	issn = {0090-5364},
	doi = {10.1214/14-AOS1252},
	number = {6},
	urldate = {2026-05-28},
	journal = {The Annals of Statistics},
	author = {Fasy, Brittany Terese and Lecci, Fabrizio and Rinaldo, Alessandro and Wasserman, Larry and Balakrishnan, Sivaraman and Singh, Aarti},
	month = dec,
	year = {2014},
}

@article{escobar_1995_BayesianDensityEstimation,
	title = {Bayesian {Density} {Estimation} and {Inference} {Using} {Mixtures}},
	volume = {90},
	issn = {0162-1459, 1537-274X},
	doi = {10.1080/01621459.1995.10476550},
	language = {en},
	number = {430},
	urldate = {2026-05-28},
	journal = {Journal of the American Statistical Association},
	author = {Escobar, Michael D. and West, Mike},
	month = jun,
	year = {1995},
	pages = {577--588},
}

@inproceedings{eldridge_2015_HartiganConsistencyMerge,
	address = {Paris, France},
	series = {Proceedings of {Machine} {Learning} {Research}},
	title = {Beyond {Hartigan} {Consistency}: {Merge} {Distortion} {Metric} for {Hierarchical} {Clustering}},
	volume = {40},
	url = {https://proceedings.mlr.press/v40/Eldridge15.html},
	booktitle = {Proceedings of the 28th Conference on Learning Theory},
	publisher = {PMLR},
	author = {Eldridge, Justin and Belkin, Mikhail and Wang, Yusu},
	editor = {Gr{\"u}nwald, Peter and Hazan, Elad and Kale, Satyen},
	month = jul,
	year = {2015},
	pages = {588--606},
}

@article{efron_1979_BootstrapMethodsAnother,
	pages = {1--26},
	title = {Bootstrap {Methods}: {Another} {Look} at the {Jackknife}},
	volume = {7},
	issn = {0090-5364},
	shorttitle = {Bootstrap {Methods}},
	doi = {10.1214/aos/1176344552},
	number = {1},
	urldate = {2026-05-28},
	journal = {The Annals of Statistics},
	author = {Efron, B.},
	month = jan,
	year = {1979},
}

@book{edelsbrunner_2010_ComputationalTopologyIntroduction,
	address = {Providence, R.I},
	title = {Computational {Topology}: {An} {Introduction}},
	isbn = {978-0-8218-4925-5},
	shorttitle = {Computational topology},
	publisher = {American Mathematical Society},
	author = {Edelsbrunner, Herbert and Harer, J.},
	year = {2010},
	note = {edelsbrunner\_2010\_ComputationalTopologyIntroduction
OCLC: ocn427757156},
}

@article{duong_2008_FeatureSignificanceMultivariate,
	title = {Feature Significance for Multivariate Kernel Density Estimation},
	volume = {52},
	copyright = {https://www.elsevier.com/tdm/userlicense/1.0/},
	issn = {01679473},
	url = {https://linkinghub.elsevier.com/retrieve/pii/S0167947308001503},
	doi = {10.1016/j.csda.2008.02.035},
	language = {en},
	number = {9},
	urldate = {2026-05-28},
	journal = {Computational Statistics \& Data Analysis},
	author = {Duong, Tarn and Cowling, Arianna and Koch, Inge and Wand, M. P.},
	month = may,
	year = {2008},
	pages = {4225--4242},
}

@book{cover_2005_ElementsInformationTheory,
	edition = {1},
	title = {Elements of {Information} {Theory}},
	copyright = {http://doi.wiley.com/10.1002/tdm\_license\_1.1},
	isbn = {978-0-471-24195-9 978-0-471-74882-3},
	url = {https://onlinelibrary.wiley.com/doi/book/10.1002/047174882X},
	doi = {10.1002/047174882X},
	language = {en},
	urldate = {2026-05-28},
	publisher = {Wiley},
	author = {Cover, Thomas M. and Thomas, Joy A.},
	month = sep,
	year = {2005},
}

@article{comaniciu_2002_MeanShiftRobust,
	title = {Mean Shift: A Robust Approach Toward Feature Space Analysis},
	volume = {24},
	copyright = {https://ieeexplore.ieee.org/Xplorehelp/downloads/license-information/IEEE.html},
	issn = {01628828},
	shorttitle = {Mean shift},
	doi = {10.1109/34.1000236},
	number = {5},
	urldate = {2026-05-28},
	journal = {IEEE Transactions on Pattern Analysis and Machine Intelligence},
	author = {Comaniciu, D. and Meer, P.},
	month = may,
	year = {2002},
	pages = {603--619},
}

@article{cheng_1995_MeanShiftMode,
	title = {Mean Shift, Mode Seeking, and Clustering},
	volume = {17},
	copyright = {https://ieeexplore.ieee.org/Xplorehelp/downloads/license-information/IEEE.html},
	issn = {01628828},
	url = {http://ieeexplore.ieee.org/document/400568/},
	doi = {10.1109/34.400568},
	number = {8},
	urldate = {2026-05-28},
	journal = {IEEE Transactions on Pattern Analysis and Machine Intelligence},
	author = {Cheng, Yizong},
	month = aug,
	year = {1995},
	pages = {790--799},
}

@article{chen_2015_AsymptoticTheoryDensity,
	title = {Asymptotic theory for density ridges},
	volume = {43},
	issn = {0090-5364},
	doi = {10.1214/15-AOS1329},
	number = {5},
	urldate = {2026-05-28},
	journal = {The Annals of Statistics},
	author = {Chen, Yen-Chi and Genovese, Christopher R. and Wasserman, Larry},
	month = oct,
	year = {2015},
}

@article{chazal_2018_RobustTopologicalInference,
	title = {Robust Topological Inference: Distance to a Measure and Kernel Distance},
	volume = {18},
	url = {https://jmlr.org/papers/v18/15-484.html},
	number = {159},
	journal = {Journal of Machine Learning Research},
	author = {Chazal, Fr{\'e}d{\'e}ric and Fasy, Brittany and Lecci, Fabrizio and Michel, Bertrand and Rinaldo, Alessandro and Wasserman, Larry},
	year = {2018},
	pages = {1--40},
}

@article{chazal_2011_GeometricInferenceProbability,
	title = {Geometric {Inference} for {Probability} {Measures}},
	volume = {11},
	issn = {1615-3375, 1615-3383},
	doi = {10.1007/s10208-011-9098-0},
	language = {en},
	number = {6},
	urldate = {2026-05-28},
	journal = {Foundations of Computational Mathematics},
	author = {Chazal, Fr{\'e}d{\'e}ric and Cohen-Steiner, David and M{\'e}rigot, Quentin},
	month = dec,
	year = {2011},
	pages = {733--751},
}

@article{chaudhuri_2000_ScaleSpaceView,
	title = {Scale Space View of Curve Estimation},
	pages = {408--428},
	volume = {28},
	issn = {0090-5364},
	doi = {10.1214/aos/1016218224},
	number = {2},
	urldate = {2026-05-28},
	journal = {The Annals of Statistics},
	author = {Chaudhuri, Probal and Marron, J. S.},
	month = apr,
	year = {2000},
}

@article{chaudhuri_1999_SiZerExplorationStructures,
	title = {{SiZer} for {Exploration} of {Structures} in {Curves}},
	volume = {94},
	issn = {0162-1459, 1537-274X},
	doi = {10.1080/01621459.1999.10474186},
	language = {en},
	number = {447},
	urldate = {2026-05-28},
	journal = {Journal of the American Statistical Association},
	author = {Chaudhuri, Probal and Marron, J. S.},
	month = sep,
	year = {1999},
	pages = {807--823},
}

@article{chaudhuri_2014_ConsistentProceduresCluster,
	title = {Consistent {Procedures} for {Cluster} {Tree} {Estimation} and {Pruning}},
	volume = {60},
	issn = {0018-9448, 1557-9654},
	doi = {10.1109/TIT.2014.2361055},
	number = {12},
	urldate = {2026-05-28},
	journal = {IEEE Transactions on Information Theory},
	author = {Chaudhuri, Kamalika and Dasgupta, Sanjoy and Kpotufe, Samory and von Luxburg, Ulrike},
	month = dec,
	year = {2014},
	pages = {7900--7912},
}

@inproceedings{chaudhuri_2010_RatesConvergenceCluster,
	address = {Red Hook, NY},
	title = {Rates of Convergence for the Cluster Tree},
	booktitle = {Advances in Neural Information Processing Systems},
	volume = {23},
	isbn = {978-1-61782-380-0},
	publisher = {Curran Associates Inc.},
	author = {Chaudhuri, Kamalika and Dasgupta, Sanjoy},
	year = {2010},
	pages = {343--351},
	url = {https://proceedings.neurips.cc/paper_files/paper/2010/hash/b534ba68236ba543ae44b22bd110a1d6-Abstract.html},
}

@article{chacon_2013_DatadrivenDensityDerivative,
	title = {Data-Driven Density Derivative Estimation, With Applications to Nonparametric Clustering and Bump Hunting},
	volume = {7},
	issn = {1935-7524},
	doi = {10.1214/13-EJS781},
	urldate = {2026-05-28},
	journal = {Electronic Journal of Statistics},
	author = {Chac{\'o}n, Jos{\'e} E. and Duong, Tarn},
	month = jan,
	year = {2013},
	pages = {499--532},
}

@incollection{carreira-perpinan_2003_NumberModesGaussian,
	address = {Berlin, Heidelberg},
	title = {On the {Number} of {Modes} of a {Gaussian} {Mixture}},
	volume = {2695},
	isbn = {978-3-540-40368-5 978-3-540-44935-5},
	doi = {10.1007/3-540-44935-3_44},
	language = {en},
	urldate = {2026-05-28},
	booktitle = {Scale {Space} {Methods} in {Computer} {Vision}},
	publisher = {Springer Berlin Heidelberg},
	author = {Carreira-Perpi{\~n}{\'a}n, Miguel A. and Williams, Christopher K. I.},
	editor = {Griffin, Lewis D. and Lillholm, Martin},
	year = {2003},
	pages = {625--640},
}

@article{carlsson_2009_TopologyData,
	title = {Topology and Data},
	volume = {46},
	issn = {1088-9485, 0273-0979},
	doi = {10.1090/S0273-0979-09-01249-X},
	language = {en},
	number = {2},
	urldate = {2026-05-28},
	journal = {Bulletin of the American Mathematical Society},
	author = {Carlsson, Gunnar},
	month = jan,
	year = {2009},
	pages = {255--308},
}

@article{campello_2015_HierarchicalDensityEstimates,
	title = {Hierarchical {Density} {Estimates} for {Data} {Clustering}, {Visualization}, and {Outlier} {Detection}},
	volume = {10},
	issn = {1556-4681, 1556-472X},
	doi = {10.1145/2733381},
	language = {en},
	number = {1},
	urldate = {2026-05-28},
	journal = {ACM Transactions on Knowledge Discovery from Data},
	author = {Campello, Ricardo J. G. B. and Moulavi, Davoud and Zimek, Arthur and Sander, Jörg},
	month = jul,
	year = {2015},
	pages = {1--51},
}

@article{botev_2010_KernelDensityEstimation,
	title = {Kernel density estimation via diffusion},
	volume = {38},
	issn = {0090-5364},
	doi = {10.1214/10-AOS799},
	number = {5},
	urldate = {2026-05-28},
	journal = {The Annals of Statistics},
	author = {Botev, Z. I. and Grotowski, J. F. and Kroese, D. P.},
	month = oct,
	year = {2010},
}

@article{biroli_2026_KernelDensityEstimators,
	title = {Kernel {Density} {Estimators} in {Large} {Dimensions}},
	volume = {8},
	issn = {2577-0187},
	doi = {10.1137/24M1703677},
	language = {en},
	number = {1},
	urldate = {2026-05-28},
	journal = {SIAM Journal on Mathematics of Data Science},
	author = {Biroli, Giulio and Mézard, Marc},
	month = mar,
	year = {2026},
	pages = {46--76},
}

@article{bickel_1973_GlobalMeasuresDeviations,
	title = {On {Some} {Global} {Measures} of the {Deviations} of {Density} {Function} {Estimates}},
	volume = {1},
	issn = {0090-5364},
	doi = {10.1214/aos/1176342558},
	number = {6},
	urldate = {2026-05-28},
	journal = {The Annals of Statistics},
	author = {Bickel, P. J. and Rosenblatt, M.},
	month = nov,
	year = {1973},
}

@book{bakry_2014_AnalysisGeometryMarkov,
	address = {Cham},
	series = {Grundlehren der mathematischen {Wissenschaften}},
	title = {Analysis and {Geometry} of {Markov} {Diffusion} {Operators}},
	volume = {348},
	isbn = {978-3-319-00226-2 978-3-319-00227-9},
	doi = {10.1007/978-3-319-00227-9},
	language = {en},
	urldate = {2026-05-28},
	publisher = {Springer International Publishing},
	author = {Bakry, Dominique and Gentil, Ivan and Ledoux, Michel},
	year = {2014},
}

@article{babaud_1986_UniquenessGaussianKernel,
	title = {Uniqueness of the {Gaussian} {Kernel} for {Scale}-{Space} {Filtering}},
	volume = {PAMI-8},
	issn = {0162-8828, 2160-9292},
	doi = {10.1109/TPAMI.1986.4767749},
	number = {1},
	urldate = {2026-05-28},
	journal = {IEEE Transactions on Pattern Analysis and Machine Intelligence},
	author = {Babaud, Jean and Witkin, Andrew P. and Baudin, Michel and Duda, Richard O.},
	month = jan,
	year = {1986},
	pages = {26--33},
}

@article{azzalini_2007_ClusteringNonparametricDensity,
	title = {Clustering via nonparametric density estimation},
	volume = {17},
	issn = {0960-3174, 1573-1375},
	doi = {10.1007/s11222-006-9010-y},
	language = {en},
	number = {1},
	urldate = {2026-05-28},
	journal = {Statistics and Computing},
	author = {Azzalini, Adelchi and Torelli, Nicola},
	month = mar,
	year = {2007},
	pages = {71--80},
}

@article{anderson_1982_ReversetimeDiffusionEquation,
	title = {Reverse-time diffusion equation models},
	volume = {12},
	issn = {03044149},
	doi = {10.1016/0304-4149(82)90051-5},
	language = {en},
	number = {3},
	urldate = {2026-05-28},
	journal = {Stochastic Processes and their Applications},
	author = {Anderson, Brian D.O.},
	month = may,
	year = {1982},
	pages = {313--326},
}

@article{ameijeiras-alonso_2019_ModeTestingCritical,
	title = {Mode Testing, Critical Bandwidth and Excess Mass},
	volume = {28},
	issn = {1133-0686, 1863-8260},
	doi = {10.1007/s11749-018-0611-5},
	language = {en},
	number = {3},
	urldate = {2026-05-28},
	journal = {TEST},
	author = {Ameijeiras-Alonso, Jose and Crujeiras, Rosa M. and Rodr{\'i}guez-Casal, Alberto},
	month = sep,
	year = {2019},
	pages = {900--919},
}

@article{abramson_1982_BandwidthVariationKernel,
	title = {On {Bandwidth} {Variation} in {Kernel} {Estimates}-{A} {Square} {Root} {Law}},
	volume = {10},
	issn = {0090-5364},
	doi = {10.1214/aos/1176345986},
	number = {4},
	urldate = {2026-05-28},
	journal = {The Annals of Statistics},
	author = {Abramson, Ian S.},
	month = dec,
	year = {1982},
}

@book{vandervaart_1996_WeakConvergenceEmpirical,
	address = {New York},
	series = {Springer series in statistics},
	title = {Weak Convergence and Empirical Processes: With Applications to Statistics},
	isbn = {978-0-387-94640-5},
	doi = {10.1007/978-1-4757-2545-2},
	url = {https://link.springer.com/book/10.1007/978-1-4757-2545-2},
	publisher = {Springer},
	author = {van der Vaart, A. W. and Wellner, Jon A.},
	year = {1996},
}

@book{villani_2009_OptimalTransportOld,
	address = {Berlin},
	series = {Grundlehren der mathematischen {Wissenschaften}},
	title = {Optimal {Transport}: {Old} and {New}},
	isbn = {978-3-540-71049-3},
	shorttitle = {Optimal transport},
	number = {338},
	publisher = {Springer},
	author = {Villani, Cédric},
	year = {2009},
	note = {villani\_2009\_OptimalTransportOld
OCLC: ocn244421231},
}

@article{edelsbrunner_2002_TopologicalPersistenceSimplification,
	title = {Topological {Persistence} and {Simplification}},
	volume = {28},
	issn = {0179-5376, 1432-0444},
	url = {http://link.springer.com/10.1007/s00454-002-2885-2},
	doi = {10.1007/s00454-002-2885-2},
	language = {en},
	number = {4},
	urldate = {2022-08-28},
	journal = {Discrete \& Computational Geometry},
	author = {Edelsbrunner, Herbert and Letscher, David and Zomorodian, Afra},
	month = nov,
	year = {2002},
	pages = {511--533},
}

@article{cohensteiner_2007_StabilityPersistenceDiagrams,
	title = {Stability of {Persistence} {Diagrams}},
	volume = {37},
	number = {1},
	journal = {Discrete \& Computational Geometry},
	author = {Cohen-Steiner, David and Edelsbrunner, Herbert and Harer, John},
	year = {2007},
	pages = {103--120},
	doi = {10.1007/s00454-006-1276-5},
}

@inproceedings{cohensteiner_2006_VinesVineyards,
	address = {New York, NY},
	title = {Vines and {Vineyards} by {Updating} {Persistence} in {Linear} {Time}},
	booktitle = {Proceedings of the Twenty-Second Annual Symposium on Computational Geometry},
	author = {Cohen-Steiner, David and Edelsbrunner, Herbert and Morozov, Dmitriy},
	year = {2006},
	pages = {119--126},
	publisher = {ACM},
	doi = {10.1145/1137856.1137877},
}

@article{carlsson_2010_ZigzagPersistence,
	title = {Zigzag {Persistence}},
	volume = {10},
	number = {4},
	journal = {Foundations of Computational Mathematics},
	author = {Carlsson, Gunnar and de Silva, Vin},
	year = {2010},
	pages = {367--405},
	doi = {10.1007/s10208-010-9066-0},
}

@article{steinwart_2023_AdaptiveClusteringKDE,
	title = {Adaptive {Clustering} {Using} {Kernel} {Density} {Estimators}},
	volume = {24},
	number = {275},
	journal = {Journal of Machine Learning Research},
	author = {Steinwart, Ingo and Sriperumbudur, Bharath K. and Thomann, Philipp},
	year = {2023},
	pages = {1--56},
	url = {https://www.jmlr.org/papers/v24/21-1532.html},
}

@article{mclachlan_2019_FiniteMixtureModels,
	title = {Finite {Mixture} {Models}},
	volume = {6},
	issn = {2326-8298, 2326-831X},
	url = {https://www.annualreviews.org/doi/10.1146/annurev-statistics-031017-100325},
	doi = {10.1146/annurev-statistics-031017-100325},
	language = {en},
	number = {1},
	urldate = {2022-02-22},
	journal = {Annual Review of Statistics and Its Application},
	author = {McLachlan, Geoffrey J. and Lee, Sharon X. and Rathnayake, Suren I.},
	month = mar,
	year = {2019},
	note = {mclachlan\_2019\_FiniteMixtureModels},
	pages = {355--378},
}

@article{chacon_2020_ModalAgeStatistics,
  author  = {Chac{\'o}n, Jos{\'e} E.},
  title   = {The Modal Age of Statistics},
  journal = {International Statistical Review},
  year    = {2020},
  volume  = {88},
  number  = {1},
  pages   = {122--141},
  doi     = {10.1111/insr.12340}
}

@article{ameijeiras-alonso_2021_multimode,
  author  = {Ameijeiras-Alonso, Jose and Crujeiras, Rosa M. and Rodr{\'i}guez-Casal, Alberto},
  title   = {{multimode}: An {R} Package for Mode Assessment},
  journal = {Journal of Statistical Software},
  year    = {2021},
  volume  = {97},
  number  = {9},
  pages   = {1--32},
  doi     = {10.18637/jss.v097.i09}
}
